%% file: Separation_from_order_and_oriented_maps_output.tex
\documentclass{article}
\usepackage{arxiv}
\input{header}
\input{preamble.sty}
\title{Convexity from order: a theory of faces and generalized separation}
\author{
  Matthew S. Scott \\
  Department of Mathematics \\
  University of British Columbia \\
  Vancouver, BC, Canada \\
  \texttt{matthewscott@math.ubc.ca}
}
\begin{document}
\maketitle
\begin{abstract}
It is known that any convex cone in a general vector space is the positive cone of a preorder compatible with the vector space,
and that any convex set is an affine slice of a convex cone.
We therefore treat convex sets order-theoretically, as affine slices of positive cones.
We find that faces are sets of strong minima,
and we generalize functionals to oriented maps valued in arbitrary ordered vector
spaces. Oriented maps expose every face, and an oriented affine map $f$
separates $C$ from $D$ when $f(D) \le 0 \le f(C)$ element-wise.
This generalization of hyperplane separation recovers mathematical objects that
are ``local", including the generated face at a point and the normal cone. The kernels of separating oriented
maps form a lattice of separating affine flats whose bottom element is the joint
supporting subspace (JSS), recently introduced in finite dimensions for
qualification-free convex analysis results.
We define the JSS in general vector spaces
for two convex sets $C, D$ as the affine slice of the cross-lineality $\lin(K-G)$,
where $K,G$ are the homogenization cones of $C, D$.
We show that in finite dimensions, bilateral facial reduction corresponds to a descent 
in the lattice of separating affine flats, which we formalize as lexicographic
products of oriented maps. The lexicographic characterization of faces follows as a corollary.

By making no assumption other than convexity, our approach yields general results: we
show that the face of a difference of two convex sets is a difference of faces, dropping
assumptions of compactness, finite dimensionality, and exposedness.
We also show that any face of an intersection of two convex sets is an intersection
of faces, resolving open problem 5.10 by Weis in ``A note on faces of convex sets",
dropping the assumption of non-empty intrinsic core.
\end{abstract}
% \msc{52A05 (Primary), 06F20, 90C25 (Secondary)}
\section{Introduction}
\label{loc:body.introduction}
There has recently been renewed interest in understanding convex sets in general vector spaces without topology~\cite{novoRelativelySolidConvex2021, millanIntrinsicCoreMinimal2023, weisNoteFacesConvex2025}.
In these works, assumptions of non-empty intrinsic core (also known as the algebraic relative
interior) are used to guarantee separation via
the algebraic Hahn-Banach theorem, or to guarantee the existence of a point generating 
a face~\cite{weisFaceGeneratedPoint2021}.
Such assumptions are far from harmless: in infinite dimensions, convex
sets with empty interior are commonplace, arising for instance from
positivity constraints in function spaces~\cite[Example 2]{weisNoteFacesConvex2025}.

Can anything be said about the geometry of convex sets in general vector spaces,
with no assumptions other than convexity?
A hint for the affirmative comes from the theory of ordered vector spaces. 
Preorders that are compatible with the vector space structure 
(preserved under translation and positive scalar multiplication) are known
to be characterized by their positive cone, which may be any convex cone.
It was shown by~\textcite{lewisFacialReductionPartially1994} that
a convex cone that is the positive cone of a vector lattice has all faces identified by
intersections with the order ideals of that lattice.
More recently, \textcite{bruynTensorProductsConvex2020} shows this fact 
for general convex cones, showing additional structure in their relation
with positive maps (order preserving linear maps) from the theory of
ordered vector spaces~\cite{bonsallSublinearFunctionalsIdeals1954}. 
The comparative generality of this theory relative to analogues for
convex sets signals a gap in the literature, made glaring by the close relationship of the two objects.
Our first contribution is to fill this gap by extending the order-theoretic theory of faces from convex cones to convex sets, by using the fact that arbitrary convex sets can be lifted
to slices of convex cones in higher dimensions, which are known as
``homogenization cones" in the literature~\cite{bauschkeHomogenizationConePolar2023}.
We use the term ``dehomogenization" to describe the procedure of inducing results for convex sets by modelling them as slices of convex cones.
Our approach has the benefit of importing the semantics of order
to describe the geometry of convex sets, 
which proves to be powerful even while describing the a priori purely
algebraic objects that are convex sets.
For example, in~\Cref{loc:characterization_of_the_faces_of_convex_sets.statement}
we characterize the non-empty faces of convex sets as sets of
minima for arbitrary compatible preorders in the vector space.
Specifically, we use a strong notion of minima, where an element is in the
minimum set of $C$ when it is dominated by all other elements in $C$.

More fundamentally, the structure from which our results emerge is an alternative axiomatization of convex sets: convex cones can be defined by their being positive cones of compatible preorders, and convex sets as affine slices of these positive cones. Convex sets can thus be defined via compatible preorders instead of convex combinations. The convexity of cones, which algebraically is closure under addition, is instead captured in the transitivity of the compatible preorder. This axiomatization offers a perspective where convex sets are ``positive regions" in affine spaces. From this perspective, the application of order-theoretic language to convex geometry becomes very natural.

For our central tool we introduce the oriented linear map, which is a linear map $\psi:W \to (Y, Y_+)$
from a vector space $W$ to an ordered vector space $(Y, Y_+)$,
where $Y_+ \subseteq Y$ is a pointed convex cone specifying a partial order in $Y$.
Oriented maps are a variation on positive maps from the theory of ordered vector
spaces, differing only in the absence of a specified order in the domain.
This variation allows for dehomogenization by restricting an oriented map to an
affine subspace $U \subseteq W$, which yields an oriented affine map $\psi|_U:U \to (Y, Y_+)$.
The oriented affine maps that are positive on a given convex set then play a role similar
to that of the positive maps for the positive cone in an ordered vector space.

Functionals $\psi:W \to \mathbb{R}$ in convex geometry serve to
expose faces of convex sets and also to separate convex sets
with hyperplanes given by their level sets.
They can themselves be understood as oriented maps, since the oriented map 
$\psi:W \to (\mathbb{R}, \mathbb{R}_+)$ is a functional with the canonical
order in $\mathbb{R}$ made explicit. 
We thus extend the class of ``exposing" maps from functionals to oriented maps,
which, as we show in~\Cref{loc:characterization_of_the_faces_of_convex_sets.statement}, make any face exposed. The classical
``exposed" faces are then those faces that have an oriented functional
among their exposing oriented maps. One way to understand the present work is as the theory resulting from extending the class of functionals to maps with codomains of higher dimensions.

Classically, a functional $\psi:X \to \mathbb{R}$ separates two convex sets $C, D \subseteq X$
when their set-images are pointwise ordered in the codomain:
\begin{equation*}
\forall c \in C, d \in D,  \psi(d) \le \psi(c),
\end{equation*}
which we write as $\psi(D) \le \psi(C)$. Then there exists $b \in \mathbb{R}$ such that $\psi(D) \le b \le \psi(C)$, and $\psi^{-1}(b)$ gives a
hyperplane separating $C$ from $D$.
By letting $\psi$ be instead a general oriented map, the above formula
becomes a generalized notion of separation.
We adopt the convention of requiring $b = 0$ with $\psi$ an affine map, in which case an affine
map $\psi:X \to (Y, Y_+)$ separates $C$ from $D$ when $\psi(D) \le 0 \le \psi(C)$.
We also make no requirement that $C$ and $D$ be disjoint, allowing 
for a theory that turns out to be richest in the case where two convex sets intersect on their
boundaries. The kernels of separating oriented affine maps form a class
of affine flats which we call separating flats, which generalize separating hyperplanes.
We show separating flats admit a number of independent characterizations in \Cref{loc:characterization_of_the_separating_flats.statement}.
The separating flats form a lattice with a bottom element, which is non-trivial
when the two convex sets intersect.
We show that this minimum separating flat is the joint supporting subspace (JSS), recently introduced in~\cite{scottBilateralFacialReduction2026},
where it is used to specify qualification-free generalizations of core duality results in convex analysis. 

Through our cone-first approach of dehomogenization, we find that the
JSS is an affine slice of $\lin(K-G)$, an object that we name the cross-lineality.
The recognition of the role of the cross-lineality is novel even in finite
dimensions, and extends the concept of the JSS to general vector spaces. 
We recover additional related results from~\cite{scottBilateralFacialReduction2026} 
with greater motivation, generality, and simpler proofs, including
the iterative procedure of bilateral facial reduction that computes the JSS in finite dimensions~\cite{scottBilateralFacialReduction2026}.
Bilateral facial reduction was independently
discovered in~\cite{linFacialReductionNice2025} where the procedure is shown to identify a ``minimal pair of faces" with guarantees that some qualification conditions hold. This pair of faces turns out to be the
intersections of the JSS with the two convex sets.
We show that these faces are indeed a ``minimal pair" in a precise sense:
they are the faces of the two convex sets $C,D$ that are generated by the intersection $C \cap D$.
This motivates our definition of the generated face $F_C(D)$ (\Cref{loc:characterization_of_the_generated_face_of_a_convex_set.statement}), where $D$
is an arbitrary convex set not necessarily contained in $C$. We thereby generalize the
classical concept of a face generated at a point,
which was recently explored as a tool in a distinct line
of work for the analysis of geometry of convex sets in general vector
spaces~\cite{weisFaceGeneratedPoint2021, weisNoteFacesConvex2025, millanIntrinsicCoreMinimal2023}.
We show that the face generated at a point is recovered from our theory of separation, when separating a
singleton $\{x\}$ from a convex set $C$ that contains it. In this special case,
the connection of the generated face with the JSS recovers \cite[proposition 2.4]{millanIntrinsicCoreMinimal2023},
a recent characterization of the
generated face. We further show in~\Cref{loc:characterization_of_the_generated_face_of_a_convex_set.statement} that for these two convex sets (and indeed for any
pair of convex sets), there is a finest
separating preorder, and that the generated face $F_C(D)$ is characterized as
the minimum set in $C$ (those points in $C$ dominated by every other in $C$)
in the finest preorder separating $C$ from $D$.

The separation of an element from its containing convex set seems to
subsume the concept of ``locality" in convex geometry.
For example, the normal cone $N_C(x)$ 
corresponds to the set of oriented maps with codomain $(\mathbb{R}, \mathbb{R}_-)$
that separate $C$ from $\{x\}$. 
Similarly, we show in~\Cref{loc:body.applications.facial_reduction_and_lexicographic_maps} that the explicitly local concepts of
supporting subspaces and nested normals in~\cite{scottBilateralFacialReduction2026} are
recovered from separation.

Using our order-theoretic tools, we derive three novel geometric results, each drawing on a different part of the framework:
\begin{itemize}
\item \emph{Faces of intersections.} Every face of $C \cap D$ is the intersection of a face of $C$ and a face of $D$. By showing this fact, we resolve open problem 5.10 posed by~\textcite{weisNoteFacesConvex2025}, by dropping the assumption of non-empty intrinsic core. This assumption was required in~\cite{weisNoteFacesConvex2025} because of the use of generated faces at a point: the non-empty intrinsic core guarantees the existence of a point generating the given face, which can then be used to generate faces for the individual convex sets $C,D$. Our theory of generated faces sidesteps this difficulty by instead obtaining generated faces as minimum sets in the finest separating preorder. The crucial step is in recognizing that the intersection of faces can be obtained by the intersection of the associated finest separating preorders.
\item \emph{Faces of differences and sums.} Every face of the difference $C - D$ is the difference of a face of $C$ and a face of $D$ (\Cref{loc:decomposition_of_faces_of_the_difference_of_convex_sets.statement}), and correspondingly for Minkowski sums. This removes the finite-dimensionality, compactness, and exposedness assumptions of~\cite[Theorem 1.7.5]{schneider_ConvexBodiesBrunn_2013}, though that result was presented in the context of a larger theory, with no particular emphasis on generality. Our result follows from the elementary fact that, in any fixed preorder, $\min (C-D) = \min(C) - \max(D)$.
\item \emph{Lexicographic characterization of faces.} We formalize bilateral facial reduction introduced in~\cite{scottBilateralFacialReduction2026} as a procedure iterating over separating oriented maps, and show that the lexicographic characterization of faces of~\textcite{martinez-legaz_LEXICOGRAPHICALCHARACTERIZATIONFACES_} follows as a corollary, bridging two previously distinct literatures.
\end{itemize}

The original motivation of this theory was to explain the structure of the JSS, establish it as a canonical object, and to allow its extension to 
general vector spaces by removing artificial assumptions. We achieved this goal in~\Cref{loc:affine_cross:lineality.statement}, providing both an intrinsic characterization as the dehomogenization of the cross-lineality, and an extrinsic one as the minimum separating flat. We arrive at this result
by combining a small number of ingredients: dehomogenization, separation, minimality in the separation lattice, and the use of oriented maps. There is very little choice in these ingredients; for example, as we argue in~\Cref{loc:body.separation}, even the consideration of separation emerges from the dehomogenization of a single general convex cone.  
Each ingredient forms an orthogonal axis along which other mathematical objects are found. We opted to provide a comprehensive description of the emerging landscape, characterizing each object along the way. In this manner we build a unifying framework, which in retrospect we consider to be our core contribution. 
For each object we provide multiple characterizations. Some characterizations specify relations to other, immediately neighbouring objects in the grid. Others bridge with definitions in the literature, sometimes in many different traditions. And others finally are novel, aiming at intuitive clarity and simplicity.
We describe the four orthogonal axes in greater detail, since they organize the contents of this work. 
One axis is homogenization, between convex cones and convex sets; we provide our results both for convex cones and convex sets. A second axis goes across oriented maps, their kernels, and the faces that they expose.
A third axis corresponds to restrictions on available
oriented maps: positivity on one convex cone (or convex set) in \Cref{loc:body.positivity_minimization_and_faces}, and then an additional
requirement of negativity on another, which gives separation, in \Cref{loc:body.separation}. 
Finally, we consider minima in the separation lattice in \Cref{loc:body.lattice_minima}, which includes the JSS and generated faces.
The objects across the four axes constitute the core of our framework, 
while the geometric results that follow in~\Cref{loc:body.applications} are applications of the framework.
\subsection{Prior works}
\label{loc:intro_prose_order_convex_sets.concise_attack.prior_works}
The use of order for the characterization of faces of convex sets 
previously appeared in the lexicographic characterization of the faces of convex sets~\cite{martinez-legaz_LEXICOGRAPHICALCHARACTERIZATIONFACES_}, where faces are characterized by the minimization of lexicographic orders. This method characterizes faces in finite dimensions, and encodes in a single order the sequential maximization of functionals.
The proof uses the fact that an exposed face of an exposed face is a face. \textcite{martinez-legaz_LEXICOGRAPHICALCHARACTERIZATIONFACES_} moreover shows that any face is attained by this recursive procedure.

The lexicographic characterization of faces was recently extended to general vector spaces~\cite{gorokhovikCharacterizationsFacesConvex2026},
where the orders are required to be total preorders. Our work shows that totality is inessential, since the minimization of any preorder yields a face.
Total orders are then
a special case --- a subset of the compatible preorders that on its own still suffices
to expose every face. This makes the total-order characterization stronger in the sense that all faces are attained from a narrower class of orders. 
The total order characterization is also the more complex, since constructing
a witnessing order for arbitrary faces requires the Kuratowski--Zorn lemma,
whereas our witnessing preorders are explicit and choice-free.

Lexicographic maps were originally used because of their unique separation properties~\cite{martinez-legazLexicographicalSeparationRn1987},
which were more recently generalized to arbitrary vector spaces via half-spaces~\cite{gorokhovikStepAffineFunctionsHalfspaces2021}.
This separation is different in kind from our
notion of separation, in that it provides a ``true" separation, containing
disjoint convex sets in two disjoint convex sets that partition the space, called half-spaces.
Unlike hyperplane separation, lexicographic separation crucially requires the disjointness of the convex sets.
Our approach is opposite: we generalize the formula of hyperplane separation
$f(D) \le f(C)$ while dropping the requirement of disjointness.
We caution that in~\Cref{loc:body.applications.facial_reduction_and_lexicographic_maps},
we derive lexicographic maps that separate two convex sets
in the sense of this work, which 
differs from lexicographic separation proper.

Vector optimization shares our central object---maps into ordered vector
spaces---but not our aim. Our minimum sets are named in \cite[Definition 2.4.4]{botDualityVectorOptimization2009} as the set of \emph{strongly minimal} elements, a stronger notion than the minimal, or Pareto-optimal, elements, which are merely undominated. The field 
sets aside the former in favour of the latter: ``For vector optimization this definition is of secondary importance because in the most practical cases strongly minimal elements do not exist"~\cite[p. 45]{botDualityVectorOptimization2009}. This caveat does
not apply to our use-case because we use preorders
as analytical tools instead of having a single preorder fixed from the problem
statement.
\section{Notation}
\label{loc:body.notation}
Throughout this work, we let the symbols $X$, $Y$ and $W$ denote general vector spaces,
with $W$ reserved for the spaces carrying convex cones, and $X$ for the spaces carrying convex sets.
A convex cone $K \subseteq W$ is a set closed under addition and positive
scalar multiplication. A convex cone may or may not contain the origin; we
adopt the convention that it does, unless
explicitly stated otherwise.
A convex cone is \emph{pointed} when it has trivial lineality $\lin(K):= K \cap -K$.

We denote $\mathbb{R}_{+} := \{x \in \mathbb{R} \mid x \ge 0\}$ and
$\mathbb{R}_{++} := \mathbb{R}_+  \setminus \{0\}$.
For $x, y \in W$, we denote by $[x, y]$ the line segment from $x$ to $y$,
which is $\{\lambda x + (1-\lambda) y \mid \lambda \in [0, 1],  x, y \in W\}$,
and by $]x, y[$ the segment with endpoints excluded.
For a set $S$, we write $\cone(S) := \{\lambda s \mid \lambda \ge 0,\ s \in S\}$, and $\convcone(S)$ for the smallest convex cone containing $S$.
For a map $f$, we denote by $f^{-1}$ its preimage, making no assumption
of invertibility. Given a set $S \subseteq \dom(f)$, we denote by
$f(S) := \{f(s)\mid s \in S\}$ the set-image.
\subsection{Compatible Preorders}
\label{loc:body.notation.compatible_preorders}
A binary relation $R \subseteq W \times W$ in a vector space $W$ is \emph{compatible} with
$W$ when it is invariant to translations and multiplication by scalars in
$\mathbb{R}_{++}$, that is,
\begin{equation*}
\forall x, y, z \in W, \lambda > 0, \quad (x,  y) \in R \implies (x + z, y+z) \in R \text{ and } (\lambda x, \lambda y) \in R.
\end{equation*}
When a binary relation is translation-invariant, we have $x \le y \iff 0 \le (y-x)$,
and the relation $R$ can be encoded in the positive cone $K := \{x \in W \mid x \ge 0\}$ (which a priori is not necessarily a cone),
since $(x, y) \in R \iff y-x \in K$. Following the convention
of ordered vector spaces, we use $K$ in place of $R$, saying 
``the binary translation-invariant relation $K \subseteq W$".
Scalar multiplication invariance further guarantees that $K$ be a cone,
one that may not contain the origin. The next result shows that the
convexity of $K$ as a cone is associated with the transitivity of the relation.
\begin{lemma}[Equivalence of convexity and transitivity]
\label{loc:equivalence_of_convexity_and_transitivity.statement}
A given compatible binary relation $K \subseteq W$ is transitive if and only if
$K$ is a convex cone, potentially without the origin.
\end{lemma}
\begin{proof}[\hypertarget{loc:equivalence_of_convexity_and_transitivity.proof}{}Proof of \Cref{loc:equivalence_of_convexity_and_transitivity.statement}]

Recall that a cone $K$ is convex when $x, y \in K \implies x + y \in K$,
and that in the relation $y < x \iff  x-y \in K$, we have $x \in K \iff x > 0$. We now show both implications in the statement for a given compatible binary relation $(>)$.

$(\implies)$:
Let $x > 0$ and $y > 0$. Translation invariance gives $x + y > y > 0$, and transitivity yields $x + y > 0$.

$(\impliedby)$:
Let $x > y > z$. By translation invariance, $x-y > 0$ and $y-z > 0$. 
Closure under addition gives $x-z = (x-y)+(y-z) > 0$, hence $x > z$.
\end{proof}
The last relevant axiom is reflexivity, which is that $0 \in K$. A reflexive,
compatible and transitive binary relation is a compatible preorder.
We also consider the non-reflexive case, where the positive cone
is a convex set without the origin. A compatible, transitive, yet non-reflexive 
binary relation is a strict order $(<)$. Each preorder $(\le)$ induces a strict
order: for a preorder $K \subseteq W$, we denote 
\begin{equation*}
x >_K y \iff x \ge_K y \; \land \; x \not \le_K y.
\end{equation*}
We also define an associated similarity relation 
\begin{equation*}
x \sim_K y \iff x \le_K y \land  x \ge_K y,
\end{equation*}
for which we have $\lin(K) = \{x \in W \mid x \sim_K 0\}$.
If a preorder $K$ is antisymmetric, in that $x \sim_K y \implies x = y$,
then $K$ is a partial order; geometrically, a pointed convex cone.

An ordered vector space is a vector space $W$ with a
distinguished preorder $K \subseteq W$, which we denote by the tuple $(W, K)$.
The choice of a single distinguished preorder is too restrictive
for our purposes: we instead consider a vector space $W$ which
contains many preorders $K \subseteq W$, none singled out.
In this work we only consider compatible preorders, and therefore any
given preorder will be understood to be compatible without explicit mention.
We extend order notation to sets elementwise: for $S, T \subseteq W$ and a
preorder $K \subseteq W$, we write $S \le_{K} T$ when
$s \le_{K} t$ for all $s \in S$ and $t \in T$, treating a vector as a
singleton. In particular, $S \ge_{K} 0$ means $S \subseteq K$.
\section{Positivity, minimization and faces}
\label{loc:body.positivity_minimization_and_faces}
A preorder can be expressed via an \emph{oriented map}.
\begin{definition}[Oriented maps]
\label{loc:oriented_maps.statement}
Let $W$ be a vector space and $(Y, Y_+)$ an ordered vector space,
where $Y_+ \subseteq Y$ is a pointed convex cone.
We define an \emph{oriented map} to be a map $\psi:W \to (Y, Y_+)$ that is linear
from $W$ to $Y$. 
\end{definition}
We say that an oriented map $\psi$ is \emph{positive} on a 
set $S \subseteq W$ when $\psi(S) \ge 0$. 
An oriented map $\psi$ can be thought of as encoding
the preorder $\psi^{-1}(Y_+)$ in its domain, the pullback of the partial order in its codomain.
\begin{proposition}[Preorders are characterized by oriented maps]
\label{loc:preorders_are_characterized_by_oriented_maps.statement}
For any preorder $P \subseteq W$, there is an oriented map $\psi:W \to (Y,Y_+)$ such that $\psi$ induces the order $P$, i.e., 
\begin{equation*}
\forall x,y \in W, x \le_{P}y \iff \psi(x) \le \psi(y),
\end{equation*}
or equivalently, $P = \psi^{-1}(Y_+)$.
One such oriented map $\psi$ inducing $P$ is the oriented quotient map 
\begin{equation*}
\pi_P:W \to (W/\lin(P),  P/\lin(P)).
\end{equation*}
\end{proposition}
To show the above result, it suffices that $P/ \lin(P)$ be a pointed convex cone
(and hence a partial order). This is guaranteed by the next lemma.
\begin{lemma}[Convex cone mapped with supporting kernel has a pointed image]
\label{loc:convex_cone_mapped_with_supporting_kernel_has_a_pointed_image.statement}
Let $K \subseteq W$ be a convex cone and $\psi:W \to Y$ a linear map. Then
\begin{equation*}
\psi(K) \text{ is pointed} \iff K \setminus \ker \psi \text{ is
a convex cone without the origin}.
\end{equation*}
\end{lemma}
\begin{proof}[\hypertarget{loc:convex_cone_mapped_with_supporting_kernel_has_a_pointed_image.proof}{}Proof of \Cref{loc:convex_cone_mapped_with_supporting_kernel_has_a_pointed_image.statement}]

Since the set $K \setminus \ker \psi$ is closed under multiplication by scalars in $\mathbb{R}_{++}$
and does not contain the origin,
by \Cref{loc:equivalence_of_convexity_and_transitivity.statement}
it is a convex cone without the origin if and only if the relation
$x <_a y  \iff y - x  \in K \setminus \ker \psi$ 
is transitive (in which case $>_a$ is a strict order).
Therefore, it suffices to show that
\begin{equation*}
(<_a) \text{ is transitive } \iff \psi(K) \text{ is a partial order }.
\end{equation*}

$(\impliedby):$
Note that by definition,
\begin{equation*}
x <_{a} y \iff \psi(x) \neq \psi(y) \land x \le_{K} y.
\end{equation*}
Now $x \leq_K y \implies \psi(x) \leq_{\psi(K)} \psi(y)$, which with $\psi(x) \neq \psi(y)$ and the partial order property of  $\psi(K)$ implies that  $\psi(x) <_{\psi(K)} \psi(y)$. This shows that
\begin{equation}
\label{eq:represent_a}
x <_{a} y \iff \psi(x) <_{\psi(K)} \psi(y) \land x \le_{K} y.
\end{equation}
Then the transitivity of $(<_a)$ holds from the transitivity of the r.h.s., 
combining the transitivity of $<_{\psi(K)}$ with the transitivity of $\leq_K$.

$(\implies):$
For $0 \leq_{\psi(K)} x \leq_{\psi(K)} 0$, there are $z_1, z_2 \in W$  such that $0 \leq_K z_1$ and $z_2 \leq_K 0$ with $\psi(z_1)=\psi(z_2)=x$. Make the assumption that $x \neq 0$, which we show results in a contradiction. Then $\psi(z_1) = \psi(z_2) \neq 0$, from which we have
$0 <_a z_1$ and $z_2 <_a 0$. By transitivity, we have $0 <_a z_1 <_a z_1 - z_2$, yet this contradicts $\psi(z_1 - z_2) = \psi(z_1)- \psi(z_2) = x-x = 0$. 
Therefore $x$ must be $0$, and this shows that $\psi(K)$ is a partial order.
\end{proof}
\begin{proof}[\hypertarget{loc:preorders_are_characterized_by_oriented_maps.proof}{}Proof of \Cref{loc:preorders_are_characterized_by_oriented_maps.statement}]

That $Y_+ := P/\lin(P)$ is a pointed cone follows from \Cref{loc:convex_cone_mapped_with_supporting_kernel_has_a_pointed_image.statement}.

That $\pi^{-1}(Y_+) = P$ follows from computing the preimage:
\begin{equation*}
\pi^{-1}(Y_+) = \pi^{-1}(\pi(P)) = P + \ker \pi = P + \lin(P) = P.
\end{equation*}
\end{proof}
We note that for a given convex cone $K \subseteq W$, and an oriented map
$\psi:W \to (Y, Y_+)$ that is positive on $K$, we can think of $(W, K)$
as a preordered vector space and $\psi:(W, K) \to (Y, Y_+)$ as an
order-preserving linear map, which is known as a positive map in the literature~\cite{bruynTensorProductsConvex2020}.
We use the notion of oriented maps because we will consider a wider
variety of sets in $W$ on which a given oriented map may or may not
be positive. Nonetheless, it is worth pointing out that positive maps
compose naturally, whereas oriented maps do not. Instead, oriented maps
can only be left-composed with positive maps, and right-composed with linear
maps.
\subsection{Positivity on convex cones}
\label{loc:body.positivity_minimization_and_faces.positivity_on_convex_cones}
Given a convex cone $K \subseteq W$ and an oriented map
$\psi:W \to (Y, Y_+)$ that is positive on $K$,
we have three induced objects of interest:
the induced preorder $\psi^{-1}(Y_+)$; the kernel $\ker \psi$; and the intersection $\ker \psi \cap K$, which will turn
out to be a face of $K$. As $\psi$ varies over all oriented maps positive on $K$,
each of these objects traces out a class which admits alternative characterizations.
The first class we have already discussed: by
\Cref{loc:preorders_are_characterized_by_oriented_maps.statement}, the induced
preorders are exactly the preorders $P$ such that $K \ge_{P} 0$.
We treat the other two: namely the faces and the supporting subspaces of $K$.
For convenience, let us fix the definition of faces from the kernels of positive oriented maps, to later show that this matches the classical definition for the faces of convex sets.
\begin{definition}[Face of convex cone]
\label{loc:face_of_convex_cone.statement}
For a convex cone $K \subseteq W$, the faces of $K$ are sets of the form 
\begin{equation*}
F = \ker \psi \cap K
\end{equation*}
for some oriented map $\psi:W \to (Y, Y_+)$ that is positive on $K$.
\end{definition}
In the following, recall that a minimum set of a convex set $S \subseteq W$ under a preorder $P \subseteq W$ is made of the elements in $S$ dominated by all of $S$, i.e.,
\begin{equation*}
\min_P S := \{ s  \in S \mid \forall s'  \in S, s \leq_{P} s'\}.
\end{equation*}
\begin{theorem}[Characterization of the faces of convex cones]
\label{loc:characterization_of_the_faces_of_convex_cones.statement}
For a convex cone $K \subseteq W$, a convex subcone $F \subseteq K$ is a face of $K$
when any of the following equivalent statements holds.
\begin{enumerate}
\item $\forall x, y \in K,  x + y \in F \implies x, y \in F.$
\item $\forall z \in F,  x \in K,  x \le_{K} z \implies x \in F$.
\item There is a preorder $P \subseteq W$ such that $K \ge_{P} 0$ and $F = \min_{P} K$.
\item There is an oriented map $\psi:W \to (Y, Y_+)$ that is positive on $K$ such that $F = K \cap \ker \psi$.
\end{enumerate}
\end{theorem}
Points 1 and 2 connect the algebraic and order-theoretic definitions of faces with \Cref{loc:face_of_convex_cone.statement} stated as point 4. Point 3 provides direct intuition about the nature of faces: faces are the sets of minima of compatible preorders. The restriction to preorders such that $K \geq 0$ serves to exclude the empty set, since from \Cref{loc:face_of_convex_cone.statement} it is not a face of convex cones. We treat convex cones as specialized objects with their own structure, for which containment of the origin is essential, and not as special cases of convex sets (for which the empty set is a face).
We also note that Point 3 directly captures the intuition that faces are ``stuck to the side" of convex cones.
\begin{proof}[\hypertarget{loc:characterization_of_the_faces_of_convex_cones.proof}{}Proof of \Cref{loc:characterization_of_the_faces_of_convex_cones.statement}]

$1 \implies 2$: For $x \in K$ and $z \in F$ with $x \le_{K} z$, there is $y \in K$ such that $x + y = z$, and so
by the splitting property $x \in F$.

$2  \implies 3$: Take the preorder $P := K - F$. Then note that $K \ge_{P} 0$ and also $F \leq_{P} K$. Further, $F$ is the full minimum set because, as we show, any point that is less than $F$ is in $F$, that is, $\forall f  \in K, f \leq_{P} F  \implies f  \in F$. We show this by decomposing the inequality of $P=K-F$.
$\forall f \in K, f \leq_{P} F  \implies \exists x  \in W: f \leq_{K} x \geq_{F} F$. The second inequality implies that $x \in F$,
therefore the expression corresponds to $f \leq_{K} F$, which by point 2 guarantees that $f  \in F$.

$3 \implies 4$: For the preorder $P$ given by point 3, take the oriented quotient map as in \Cref{loc:preorders_are_characterized_by_oriented_maps.statement}. The minimum set is $K \cap \lin(P)$, $K \ge_{P} 0$ implies that
$0 \in \min_P K$, hence the minimum is the set of points that $P$ puts in equivalence with $0$, which is $\lin(P)$. Then since $\lin(P)$ is the kernel of the quotient map, point 4 follows.

$4 \implies 1$: Consider any $x, y \in K$ such that $x + y \in F$.
Then for the oriented map $\psi$ given by point 4, we have
$\psi(x),  \psi(y) \ge 0$ because $\psi(K) \ge 0$.
Additionally, since $x + y \in F \subseteq \ker \psi$,
we also have $\psi(x) + \psi(y) =  \psi(x + y) = 0$,
i.e., $\psi(x)= -\psi(y)$. From this,
it follows that $\psi(x) \sim \psi(y) \sim 0$,
and since $Y_+$ specifies a partial order in $Y$,
$\psi(x)= \psi(y)= 0$.
This shows that $x, y \in \ker \psi \cap K = F$.
\end{proof}
Since we define faces via the intersection $K \cap \ker \psi$, we find
it of interest to consider those subspaces that take the form $\ker \psi$.
\begin{definition}[Supporting subspace]
\label{loc:supporting_subspace_definition_from_oriented_maps.statement}
A supporting subspace of a convex cone $K \subseteq W$ is a subspace of the form
\begin{equation*}
U := \ker \psi
\end{equation*}
for some oriented linear map $\psi:W \to (Y, Y_+)$ that is positive on $K$.
\end{definition}
\begin{theorem}[Characterization of supporting subspaces]
\label{loc:characterization_of_supporting_subspaces.statement}
A supporting subspace $U \subseteq W$ of a cone $K \subseteq W$ is equivalently
\begin{enumerate}
\item A subspace $U \subseteq W$ such that $K \setminus U$ is a convex cone without the origin.
\item A subspace $U \subseteq W$ such that $\forall y \in W, z \in U, 0 \le_{K} y \le_{K} z  \implies y \in U$. 
\item $U=\ker \psi$ for some oriented map $\psi:W \to (Y, Y_+)$ that is positive on $K$.
\item $U = \{ x  \in W \mid x \sim_{P} 0\}$ for a preorder $P \subseteq W$ such that $K \geq_{P} 0$.
\item $U \subseteq W$ is a subspace such that $U \cap K \triangleleft K$.
\end{enumerate}
\end{theorem}
Point 1 corresponds to the definition of supporting subspaces in \cite{scottBilateralFacialReduction2026}, point 2 is the definition of the order ideals of $K$, and point 3 describes the relation of supporting subspaces to oriented maps (or a positive map, if we fix the order in $W$ to be $K$). We note that point 4 characterizes supporting subspaces as $\lin(P)$ for suitable preorders $P$, showing that supporting subspaces can be thought of as lineality spaces.
Point 5 shows that supporting subspaces are also characterized by the faces of $K$. The closest antecedent in the literature to this multifaceted characterization is \cite{bruynTensorProductsConvex2020} Appendix A, where the equivalence between supporting hyperplanes and order ideals is shown.
\begin{proof}[\hypertarget{loc:characterization_of_supporting_subspaces.proof}{}Proof of \Cref{loc:characterization_of_supporting_subspaces.statement}]

$1  \implies 3$: Consider a linear map $\psi:W \to Y$ with kernel $U$. By \Cref{loc:convex_cone_mapped_with_supporting_kernel_has_a_pointed_image.statement}, $\psi(K)$ is a pointed cone, hence $\psi:W \to (Y, \psi(K))$ forms an oriented map that is positive on $K$, with kernel $U$.

$3 \implies 4$: This holds for the induced preorder $P := \psi^{-1}(Y_+)$.

$4 \implies 5$: Note that $U \cap K = \{ x \in K \mid x \sim_{P} 0\} = \min_P K$, hence $U \cap K \triangleleft K$ by \Cref{loc:characterization_of_the_faces_of_convex_cones.statement} point 3.

$5 \implies 2$: Consider $U$ such that $(U \cap K) \triangleleft K$.
For any $y \in W$ and $z \in U$ with $0 \le_{K} y \le_{K} z$,
we have $y \in K$ and $z \in K \cap U$, and so
by \Cref{loc:characterization_of_the_faces_of_convex_cones.statement} point 2
it follows that $y \in K \cap U$, so $y \in U$.

$2  \implies 1$: Let $x_1,x_2 \in K \setminus U$. Let $z = x_1 + x_2$.
Note that $0 \le_{K} x_1$ and $x_2 \le_{K} z$, and therefore
by point $2$, $z \in U \implies x_1, x_2 \in U$, a contradiction. 
Therefore $K \setminus U$ is closed under addition, and it is
immediate that $\forall x \in K \setminus U, \forall \lambda >0, \lambda x \in K \setminus U$.
This shows that $K\setminus U$ is a convex cone without the origin.
\end{proof}
\subsection{Dehomogenization}
\label{loc:body.positivity_minimization_and_faces.dehomogenization}
Homogenization is the operation of taking a set $S \subseteq X$ and lifting it to the cone $\cone(S \times \{ 1 \}) \subseteq X  \times \mathbb{R}$.
We consider cones as primary and sets as derived, letting the vector space $X$ be embedded as an affine slice $X  \times \{ 1 \}$ in the larger vector space $X  \times \mathbb{R}$. We denote $\gamma_1: X \to X  \times \mathbb{R}$ the map embedding $X$ to $X  \times \{ 1 \}$. A \emph{homogenization cone} is a cone $K$
that is contained in $(X \times \mathbb{R}_{++}) \cup \{ 0 \}$. We take
an affine slice of a homogenization cone with the preimage 
$S := \gamma_1^{-1}(K) \subseteq X$. Every homogenization cone is recovered from its slice.
\begin{lemma}[Dehomogenization is a bijection between sets and homogenization cones]
\label{loc:dehomogenization_is_a_bijection_between_sets_and_homogenization_cones.statement}
The set-valued map $\gamma_1^{-1}$ is a bijection from the set of homogenization
cones to the subsets of $X$. Moreover, $\gamma_1^{-1}$ has inverse $\cone\circ\gamma_1$.
\end{lemma}
We leave \hyperlink{loc:dehomogenization_is_a_bijection_between_sets_and_homogenization_cones.proof}{proof} to the appendix.

Convexity of a set in $X$ is equivalent to convexity of its homogenization cone.
\begin{lemma}[Dehomogenization is a lattice isomorphism between convex sets and homogenization cones]
\label{loc:dehomogenization_is_a_lattice_isomorphism_between_convex_sets_and_homogenization_cones.statement}
Let $C \subseteq X$ and $K := \cone\circ\gamma_1(C)$. Then
\begin{equation*}
C \text{ is convex} \iff K \text{ is convex}.
\end{equation*}
Hence the set-valued map $\gamma_1^{-1}(\cdot)$ restricts to a lattice isomorphism from the convex homogenization cones in $X  \times \mathbb{R}$ to the convex sets in $X$, with inverse $\cone \circ \gamma_1$.
\end{lemma}
\begin{proof}[\hypertarget{loc:dehomogenization_is_a_lattice_isomorphism_between_convex_sets_and_homogenization_cones.proof}{}Proof of \Cref{loc:dehomogenization_is_a_lattice_isomorphism_between_convex_sets_and_homogenization_cones.statement}]

Let $C \subseteq X$ and $K := \cone\circ\gamma_1(C)$. Any two non-zero elements of $K$ can be formulated as $\lambda(x,1)$, $\mu(y,1)$ with $x,y \in C$ and $\lambda,\mu \in \mathbb{R}_{++}$. With $t := \frac{\lambda}{\lambda+\mu} \in ]0, 1[$,
\begin{equation*}
\lambda(x,1) + \mu(y,1) = (\lambda+\mu)\bigl(tx + (1-t)y,\ 1\bigr).
\end{equation*}

$(\Rightarrow)$ If $C$ is convex, the right-hand side lies in $K$, so $K$ is closed under addition, hence convex.

$(\Leftarrow)$ If $K$ is convex and $x,y \in C$, $t \in ]0,1[$, then $(tx+(1-t)y,\ 1) = t(x,1)+(1-t)(y,1) \in K$, so $tx+(1-t)y \in C$.

By \Cref{loc:dehomogenization_is_a_bijection_between_sets_and_homogenization_cones.statement}, $\gamma_1^{-1}$ and $\cone\circ\gamma_1$ pair $C$ with $K$. The equivalence therefore restricts that bijection to convex sets and convex homogenization cones. Both sides are lattices under inclusion and $\gamma_1^{-1}$ preserves containment, so the restriction is a lattice isomorphism.
\end{proof}
\begin{definition}[Oriented affine maps]
\label{loc:oriented_affine_maps.statement}
An oriented affine map is an affine map $f:X \to (Y, Y_+)$ from a vector space $X$ to a partially ordered vector space $(Y, Y_+)$, where $Y_+ \subseteq Y$ is a partial order.
\end{definition}
Given $x, y \in X \times \{1\}$ the difference $x-y$ lies in $X \times \{0\}$.
The space $X \times \{0\}$ therefore corresponds to the tangent space
of the affine space $X \times \{1\}$. We let $\gamma_0$ be the map embedding $X$
as $X \times \{0\} \subseteq X \times \mathbb{R}$.
The dehomogenization of oriented maps takes the form of a precomposition $f := \psi \circ \gamma_1$, as this corresponds to a restriction $\psi_{X  \times \{ 1 \}}$
up to identifying $X  \times \{ 1 \}$ with $X$.
We next specify this dehomogenization of oriented maps.
\begin{proposition}[Oriented linear maps are in bijection with oriented affine maps]
\label{loc:oriented_linear_maps_are_in_bijection_with_oriented_affine_maps.statement}
Let $X$ be a vector space and $(Y, Y_+)$ a partially ordered vector space.
The following three classes of objects are in bijection.
\begin{enumerate}
\item pairs $(A, b)$ of an oriented linear map $A: X \to (Y, Y_+)$ and a vector $b \in Y$;
\item oriented affine maps $f: X \to (Y, Y_+)$;
\item oriented linear maps $\psi: X \times \mathbb{R} \to (Y, Y_+);$
\end{enumerate}

They are in bijection via the following mutually inverse relations.
\begin{equation*}
f(x) = \psi \circ \gamma_1(x)= A(x) + b,
\end{equation*}
\begin{equation*}
A(x) = \psi \circ \gamma_0(x) = f(x) - f(0) ,
\end{equation*}
\begin{equation*}
b = f(0) = \psi(0, 1),
\end{equation*}
\begin{equation*}
\psi(x, \lambda) = A(x) + \lambda b = (f(x) - f(0)) + \lambda f(0).
\end{equation*}

Further, when $\ker f \neq \varnothing$, the offset $b$ is represented by any
fixed point $x_0 \in \ker f = A^{-1}(-b)$, in that
\begin{equation*}
f(x) = A(x - x_0), \qquad \psi(x, \lambda) = A(x - \lambda x_0).
\end{equation*}
\end{proposition}
\begin{remark}

The correspondence between oriented affine maps and
oriented linear maps constitutes a universal property of the
pair $(X \times \mathbb{R}, \gamma_1)$; that any affine oriented map out of $X$ is
expressed by a linear oriented map out of $X \times \mathbb{R}$:
\begin{equation*}
\forall f:X \to (Y,Y_+) \quad  \exists!\, \psi: X \times \mathbb{R} \to (Y, Y_+) \; : \quad f =
\psi \circ \gamma_1.
\end{equation*}
\end{remark}
In \Cref{loc:oriented_linear_maps_are_in_bijection_with_oriented_affine_maps.statement},
the case where $\ker f \neq \varnothing$ allows for simpler formulas: after fixing some $x_0 \in \ker f$,
the homogenization maps $\psi:X  \times \mathbb{R} \to (Y,Y_+)$ are in correspondence with the linear maps $A:X \to (Y,Y_+)$.
The act of choosing $x_0$ in \Cref{loc:oriented_linear_maps_are_in_bijection_with_oriented_affine_maps.statement} can be understood as effectively ``picking an origin" in $X$.
Indeed, $\gamma_1$ embeds $X$ as the affine space $X  \times \{ 1 \}$, hence $X$ can itself be thought of as an affine space. Its origin plays no special role other than in specifying the product $X  \times \mathbb{R}$. The choice of a meaningful $x_0  \in X$ can thus be thought of as picking a ``real" origin for $X$, making it into a ``true" vector space.
In such a case, the now-vector-space $X$ takes the place of the homogenization space as the base vector space, allowing for simplified results.
This simplification via a ``choice of origin" in $X$ is a recurring theme in this work.
\begin{proof}[\hypertarget{loc:oriented_linear_maps_are_in_bijection_with_oriented_affine_maps.proof}{}Proof of \Cref{loc:oriented_linear_maps_are_in_bijection_with_oriented_affine_maps.statement}]

Fix $A$ and $b$, and let $f(x) := A(x)+b$, $\psi(x,\lambda) := A(x) + \lambda b$.
Then all the remaining equalities are immediate.
Further, the form $f(x) = A(x)+ b$ describes arbitrary affine maps,
and the form $A(x)+ \lambda b$ arbitrary linear maps, therefore the specified
equations specify pairwise bijections between all three objects.
Finally, when $\ker f \neq \varnothing$, 
fix $x_0 \in \ker f$, and note that
$f(x_0) = A(x_0)+b = 0$,
therefore $b = -A(x_0)$. Substituting in the equations
$f(x)= A(x)+ b$ and $\psi(x, \lambda)= A(x) + \lambda b$ yields the result.
\end{proof}
We now show a number of lemmas describing the dehomogenization of oriented maps positive on sets, the dehomogenization of kernels, and the dehomogenization
of spans.
\begin{lemma}[Preservation of orientation in dehomogenization]
\label{loc:preservation_of_orientation_in_dehomogenization.statement}
Let $C \subseteq X$ be a convex set and $K := \cone(\gamma_1(C))$ its homogenization cone. For any linear oriented map $\psi:X \times \mathbb{R} \to(Y,Y_+)$, with $f:= \psi \circ \gamma_1$,
\begin{equation*}
\psi(K) \geq 0  \iff f(C) \geq 0,
\end{equation*}
hence the bijection of \Cref{loc:oriented_linear_maps_are_in_bijection_with_oriented_affine_maps.statement} between linear and affine oriented maps
restricts to a bijection between those maps positive on $K$ and $C$ respectively.
\end{lemma}
\begin{proof}[\hypertarget{loc:preservation_of_orientation_in_dehomogenization.proof}{}Proof of \Cref{loc:preservation_of_orientation_in_dehomogenization.statement}]

$(\implies):$ $\psi(K) \geq 0  \implies \psi(\gamma_1(\gamma_1^{-1}(K))) \geq 0  \implies f(C) \geq 0$.

$(\impliedby)$: Any point in $K$ is either the origin, in which case $\psi(0) \geq 0$ trivially, or it is of the form $(x,\lambda)$, for $\lambda > 0$, and such that $\frac{1}{\lambda}(x,\lambda) = (x /\lambda, 1)  \in C  \times \{ 1 \}$, and hence $x /\lambda  \in C$. Then since
\begin{equation*}
\psi(x,\lambda) = \lambda \psi(x / \lambda, 1) = \lambda f(x / \lambda),
\end{equation*}
we have $\lambda > 0$ and $f(x /\lambda) \geq 0$ because $x /\lambda  \in C$. Therefore $\psi(x,\lambda) \geq 0$, showing that $\psi(K) \geq 0$.
\end{proof}
\begin{lemma}[Dehomogenization of kernels]
\label{loc:dehomogenization_of_kernels.statement}
Let $\psi:X  \times \mathbb{R} \to (Y,Y_+)$ with $f := \psi \circ \gamma_1$. 
Then
\begin{equation*}
\ker f = \gamma_1^{-1}(\ker \psi).
\end{equation*}
\end{lemma}
\begin{proof}[\hypertarget{loc:dehomogenization_of_kernels.proof}{}Proof of \Cref{loc:dehomogenization_of_kernels.statement}]

The identity follows from composition of preimages:
\begin{equation*}
\ker f = (\psi \circ \gamma_1)^{-1}(\{0\}) = \gamma_1^{-1}\left(\psi^{-1}(\{0\})\right) = \gamma_1^{-1}(\ker \psi).
\end{equation*}
\end{proof}
\begin{lemma}[Span dehomogenizes to the affine hull]
\label{loc:span_dehomogenizes_to_the_affine_hull.statement}
Let $S \subseteq X$ be a convex set and $H := \cone \circ \gamma_1(S)$ its homogenization cone. Then
\begin{equation*}
\gamma_1^{-1}(\Span H) = \aff(S).
\end{equation*}
\end{lemma}
We defer \hyperlink{loc:span_dehomogenizes_to_the_affine_hull.proof}{the proof} to the appendix.
\subsection{Positivity on convex sets}
\label{loc:body.positivity_minimization_and_faces.positivity_on_convex_sets}
The following are the affine analogues of supporting subspaces.
\begin{definition}[Supporting flats]
\label{loc:supporting_flats.statement}
A supporting flat $U$ of a convex set $C \subseteq X$ is
\begin{equation*}
U := \ker f
\end{equation*}
for some oriented affine map $f:X \to (Y, Y_+)$ that is positive on $C$.
\end{definition}
Similarly to the faces of convex cones, we define the faces of convex sets as intersections with kernels, and later show equivalence with the classical definition of the faces of convex sets.
\begin{definition}[Faces of convex sets]
\label{loc:faces_of_convex_sets.statement}
For a convex set $C \subseteq X$, the faces of $C$ are the sets of the form
\begin{equation*}
F = C \cap \ker f
\end{equation*}
for some affine oriented map $f:X \to (Y, Y_+)$ that is positive on $C$.
\end{definition}
The next lemma shows that the preimage $\gamma_1^{-1}(\cdot)$ dehomogenizes faces. 
\begin{lemma}[Dehomogenization is a bijection on the faces]
\label{loc:dehomogenization_is_a_bijection_on_the_faces.statement}
Let $C \subseteq X$ be a convex set and $K := \cone \circ \gamma_1(C)$ its homogenization cone.
The set-valued map $\gamma_1^{-1}(\cdot)$ restricts to a lattice
isomorphism between the faces of $K$ and the faces of $C$.
\end{lemma}
\begin{proof}[\hypertarget{loc:dehomogenization_is_a_bijection_on_the_faces.proof}{}Proof of \Cref{loc:dehomogenization_is_a_bijection_on_the_faces.statement}]

First, by \Cref{loc:preservation_of_orientation_in_dehomogenization.statement},
we have $\psi(K) \geq 0 \iff f(C) \geq 0$, over the pairs $f = \psi \circ \gamma_1$ of \Cref{loc:oriented_linear_maps_are_in_bijection_with_oriented_affine_maps.statement}. We have the identity
\begin{equation*}
\gamma_1^{-1}(K \cap \ker \psi) = \gamma_1^{-1}(K) \cap \gamma_1^{-1}(\ker \psi) = C \cap \ker f,
\end{equation*}
by the preimage commuting with intersections, and by the identity $\gamma_1^{-1}(\ker \psi)= \ker f$ from \Cref{loc:dehomogenization_of_kernels.statement}.
Since any face of $K$ takes the form of the l.h.s. (\Cref{loc:face_of_convex_cone.statement}), any face of $C$ takes the form of the r.h.s. (\Cref{loc:faces_of_convex_sets.statement}), and $\gamma_1^{-1}$ is a lattice isomorphism between convex sets and convex
cones contained in $(X \times \mathbb{R}_{++})\cup\{0\}$ (\Cref{loc:dehomogenization_is_a_lattice_isomorphism_between_convex_sets_and_homogenization_cones.statement}), the statement follows.
\end{proof}
The faces of convex sets admit the following characterizations, 
which correspond to the characterizations of the faces of convex cones in \Cref{loc:characterization_of_the_faces_of_convex_cones.statement}. In the following,
$]x, y[$ is the line segment from $x$ to $y$ excluding both endpoints.
\begin{theorem}[Characterization of the faces of convex sets]
\label{loc:characterization_of_the_faces_of_convex_sets.statement}
Let $C \subseteq X$ and $F \subseteq C$ be convex sets.
Then $F \triangleleft C$ precisely when any, hence all, of the following equivalent statements hold.
\begin{enumerate}
\item $\forall x, y \in C, ]x, y[ \cap F \neq \varnothing \implies x, y \in F$.
\item For $K := \cone\circ\gamma_1(C)$, there is a face $T  \triangleleft K$ such that $F = \gamma_1^{-1}(T)$.
\item Either $F = \varnothing$, or $F = \min_{P} C$ for some compatible preorder $P \subseteq X$.
\item $F = C \cap \ker f$ for some oriented affine map $f:X \to (Y,Y_+)$ that is positive on $C$.
\item $F = C \cap L$ for $L$ a supporting flat of $C$.
\end{enumerate}
\end{theorem}
The novel results in this list are points 3 and 4. 

Point 3 provides a direct, intuitive notion of faces as sets of minima, under the caveat of non-emptiness. Making $\varnothing$ a face explicitly is necessary because of the case where $C$ is a singleton, in which case no preorder has the empty set as its set of minima (in all other cases, the discrete preorder $\{0\}$ has an empty set of minima).

Point 4 corresponds to \Cref{loc:faces_of_convex_sets.statement}, which we choose as
our core definition due to its unconditional nature. It is the dehomogenization
of the characterization of faces by positive maps from the theory of ordered vector spaces.

Point 1 shows equivalence of the other characterizations with the classical definition of faces of convex sets, while point 5 recovers another characterization of faces when considered together with \Cref{loc:characterization_of_supporting_flats.statement} point 1, in that a face
is characterized by two facts: that it is given by intersection with an affine flat, and
that its complement is convex. This characterization of faces was first stated as an exercise in \cite{brondstedIntroductionConvexPolytopes1983}, and later shown in \cite{scottBilateralFacialReduction2026}. Our exposition explains this characterization by identifying the supporting flats with the kernels of oriented affine maps and the convexity of the complement with the positivity in the partial order of the codomain.
\begin{proof}[\hypertarget{loc:characterization_of_the_faces_of_convex_sets.proof}{}Proof of \Cref{loc:characterization_of_the_faces_of_convex_sets.statement}]

$1 \iff 2$: The equivalence  follows from convex combinations being in correspondence with sums in the homogenization space. We show this in detail.

$(1 \implies 2)$:
Any element in $K := \cone(\gamma_1(C))$ can be expressed as $(\lambda_x x, \lambda_x) \in K$, for some $\lambda_x \geq 0$, $x \in C$. Let $T = \cone(\gamma_1(F))$, a convex cone since $F$ is convex, and note that $F = \gamma_1^{-1}(T)$ by \Cref{loc:dehomogenization_is_a_lattice_isomorphism_between_convex_sets_and_homogenization_cones.statement}. 
We show that $T$ is a face of $K$ because it has the splitting property (\Cref{loc:characterization_of_the_faces_of_convex_cones.statement} point 1).
For any two elements $(\lambda_x x,  \lambda_x),  (\lambda_y y,  \lambda_y) \in K$
such that $(\lambda_x x,  \lambda_x) + (\lambda_y y,  \lambda_y) \in T$,
if either is zero, then their sum being in $T$ implies they are both in $T$ trivially. Otherwise,
let $\lambda := \frac{\lambda_x}{\lambda_x + \lambda_y}$. The sum is
\begin{equation*}
(\lambda_x x,  \lambda_x) + (\lambda_y y,  \lambda_y) = (\lambda_x + \lambda_y)(\lambda x + (1-\lambda) y, 1).
\end{equation*}
Since the l.h.s. is in $T$, so is the r.h.s. Dividing by $(\lambda_x + \lambda_y)$ on the r.h.s.,
we find an element in $T$ at height $1$, so $(\lambda x + (1-\lambda) y) \in F$. By point 1, this implies $x,y \in F$, hence $(\lambda_x x, \lambda_x), (\lambda_y y,\lambda_y) \in T$.

$(2 \implies 1)$:
Note that since $T \triangleleft K$, we have $T \subseteq K \subseteq X  \times \mathbb{R}_{++} \cup \{ 0 \}$, hence $T$ is a homogenization cone. Therefore for $F := \gamma_1^{-1}(T)$, we have $T = \cone(\gamma_1(F))$ by \Cref{loc:dehomogenization_is_a_bijection_between_sets_and_homogenization_cones.statement}. 
Consider any $x,y \in C$, $\lambda \in ]0,1[$ such that $\lambda x + (1-\lambda) y \in F$. Then note that
\begin{equation*}
(\lambda x + (1-\lambda) y, 1) = (\lambda x, \lambda) + ((1-\lambda) y, 1-\lambda),
\end{equation*}
The l.h.s. can be seen to be in $T$. Then since the sum on the r.h.s. is in $K$, the splitting property of $T$ (\Cref{loc:characterization_of_the_faces_of_convex_cones.statement} point 1) guarantees that
the individual terms are in $T$, from which it follows that $x, y \in F$.

$4 \iff 2$: Given by \Cref{loc:dehomogenization_is_a_bijection_on_the_faces.statement},
with faces defined via kernels by \Cref{loc:face_of_convex_cone.statement} for convex cones and by \Cref{loc:faces_of_convex_sets.statement} for convex sets.
$3 \iff 4$: When $F = \varnothing$, both points hold trivially.
Otherwise, in all cases we may fix some $x_0 \in F$.

$(3 \implies 4)$: We have the quotient map $\pi_P$ defined from the preorder $P$
as in \Cref{loc:preorders_are_characterized_by_oriented_maps.statement},
and then $f(x) := \pi_P(x-x_0)$ is an oriented affine map positive on $C$ such that $\ker f \cap C = \min_P C$.

$(4 \implies 3)$: The face is given by minimization in the preorder induced by $f$
($x \le_{f} y \iff f(x) \le f(y)$).
Since we have both $f(C) \ge 0$ and $f(x_0) = 0$ for some witness $x_0 \in C$, it follows that
$\min_f C = \argmin_{x \in C} f(x) = \{x \in C \mid f(x) = 0\} = C \cap  \ker f$.

$5 \iff 4$: The statements correspond, up to specifying supporting flats as in \Cref{loc:supporting_flats.statement}.
\end{proof}
Finally, we characterize the supporting flats, obtained via
dehomogenization of supporting subspaces.
\begin{theorem}[Characterization of supporting flats]
\label{loc:characterization_of_supporting_flats.statement}
Let $C \subseteq X$ be a convex set and $L \subseteq X$ an affine flat. The following are equivalent characterizations for $L$ being a supporting flat of $C$.
\begin{enumerate}
\item $C \setminus L$ is convex.
\item $L = \ker f$ for some affine oriented map $f:X \to (Y, Y_+)$ that is positive on $C$.
\item $L \cap C$ is a face of $C$.
\item $L = \gamma_1^{-1}(U)$ for $U$ a supporting subspace of $K := \cone(\gamma_1(C))$.
\end{enumerate}
\end{theorem}
Point 1 shows that the supporting flats match the notion in \cite[Definition 6.1]{scottBilateralFacialReduction2026}. Point 3 is also notable: not only do supporting flats characterize faces, but faces also characterize supporting flats, a novel characterization. The analogous fact for convex cones has a direct justification: by \Cref{loc:characterization_of_supporting_subspaces.statement} point 2, a supporting subspace $U$ is supporting when $\forall y \in W, z \in U, 0 \le_{K} y \le_{K} z  \implies y \in U$. Yet positivity in $K$ implies that all feasible points $y, z$ in the statement must be in $K$, hence the statement is concerned only with the intersection set $K \cap U$. 

To show \Cref{loc:characterization_of_supporting_flats.statement}, we require the
following lemma.
\begin{lemma}[Convex set mapped with supporting kernel has a pointed image]
\label{loc:convex_set_mapped_with_supporting_kernel_has_a_pointed_image.statement}
Given a convex set $C \subseteq X$ and an affine map $\phi:X \to Y$,
\begin{equation*}
\cone (\phi(C)) \text{ is pointed} \iff C \setminus \ker \phi \text{ is convex}.
\end{equation*}
\end{lemma}
\begin{proof}[\hypertarget{loc:convex_set_mapped_with_supporting_kernel_has_a_pointed_image.proof}{}Proof of \Cref{loc:convex_set_mapped_with_supporting_kernel_has_a_pointed_image.statement}]

For $K = \cone \circ \gamma_1(C)$ and $\psi$ the homogenization of $\phi$ given by \Cref{loc:oriented_linear_maps_are_in_bijection_with_oriented_affine_maps.statement},
\begin{align*}
\cone(\phi(C)) \text{ is pointed}
&\iff \psi(K) \text{ is pointed} \\
&\iff K \setminus \ker \psi \text{ is a convex cone without the origin} \\
&\iff C \setminus \ker \phi \text{ is convex}.
\end{align*}
The first equivalence holds because $\psi(K) = \cone(\phi(C))$, since $\psi(\lambda\gamma_1(c)) = \lambda\phi(c)$.
The second is \Cref{loc:convex_cone_mapped_with_supporting_kernel_has_a_pointed_image.statement}.
We show the third. By \Cref{loc:dehomogenization_of_kernels.statement}, $\ker\phi = \gamma_1^{-1}(\ker\psi)$, so
\begin{equation*}
\gamma_1^{-1}(K \setminus \ker\psi) = \gamma_1^{-1}(K) \setminus \gamma_1^{-1}(\ker\psi) = C \setminus \ker\phi.
\end{equation*}
Since $K \setminus \ker\psi \subseteq (X \times \mathbb{R}_{++})\cup\{0\}$, the last equivalence is given by \Cref{loc:dehomogenization_is_a_lattice_isomorphism_between_convex_sets_and_homogenization_cones.statement}, which states that a homogenization cone is convex if and only if its associated set is convex.
\end{proof}
\begin{proof}[\hypertarget{loc:characterization_of_supporting_flats.proof}{}Proof of \Cref{loc:characterization_of_supporting_flats.statement}]

$1 \implies 2$: Consider any affine map $f:X \to Y$ with kernel $L$. We provide it with orientation $f:X \to (Y, \cone(f(C)))$. By \Cref{loc:convex_set_mapped_with_supporting_kernel_has_a_pointed_image.statement},
$\cone(f(C))$ is pointed, hence we have a valid oriented map.
$2 \implies 3$: That $L \cap C$ is a face holds by our definition of faces, \Cref{loc:faces_of_convex_sets.statement}.

$3  \implies 1$: Since $L \cap C$ is a face, by \Cref{loc:faces_of_convex_sets.statement} there exists an affine oriented map $f$ such that $\ker f \cap C = L \cap C$ (making no claim that $\ker f = L$ here). Then $\ker f$ is a supporting flat, and by the above argument $(2  \implies 1)$, $C \setminus \ker f$ is convex, hence $C\setminus L$ is convex by virtue of being the same set.

$2 \iff 4$: The kernels correspond as $\ker f = \gamma_1^{-1}(\ker \psi)$ by \Cref{loc:dehomogenization_of_kernels.statement}, and $f$ is positive on $C$ exactly when $\psi$ is positive on $K$ by \Cref{loc:preservation_of_orientation_in_dehomogenization.statement}. This relates to the supporting subspaces in point 4 because supporting subspaces are the kernels of the oriented maps $\psi$ positive on $K$ by \Cref{loc:supporting_subspace_definition_from_oriented_maps.statement}. 
\end{proof}
\section{Separation}
\label{loc:body.separation}
Let $K, G \subseteq W$ be convex cones.
We say that a preorder $P \subseteq W$ separates $K$ from $G$ when
$G \le_{P} K$.
There always exists a preorder separating two convex cones (for instance, the
preorder obtained by the set difference $K-G:=\{ k - g \mid k \in K,g  \in G\}$). 
An important realization is that in any given preorder,
\begin{equation*}
G \le_{P} K  \iff  G \leq_{P} 0 \leq_{P} K,
\end{equation*}
because the origin is contained in any convex cone, hence $G \leq_{P} K  \implies 0 \leq_{P} K$ from the fact that $0 \in G$.
This statement strengthens to $G \leq_{P} \lin(P) \leq_{P} K$ since
$\lin(P) = \{ x  \in W \mid x \sim_{P} 0\}$, so there naturally
emerges a subspace $\lin(P)$ induced by the preorder separating $K$ and $G$.

Any preorder $P \subseteq W$ has three zones of importance in $W$:
the positive cone $P$, the negative cone $-P$, and the lineality $\lin(P)$.
We derived above that two convex cones being ordered by $P$ implies
the stronger notion that one convex cone is contained in the positive cone $P$ and the other in the negative cone $-P$,
with $\lin(P)$ between them.

Expressing the preorder $P$ via a corresponding oriented linear map
$\psi:W \to (Y, Y_+)$ (for which $P = \psi^{-1}(Y_+)$), separation takes the form
$\psi(G)\le \psi(K)$, with the equivalent statement $\psi(G) \leq 0 \leq \psi(K)$.
The subspace in between the two is $U = \ker \psi$.
When $\psi$ is a functional, i.e. its codomain is the ordered real line $(\mathbb{R}, \mathbb{R}_+)$, $U$ is a
separating hyperplane between $K$ and $G$. The kernels of separating oriented
maps therefore generalize separating hyperplanes, and we call them \emph{separating subspaces}.

There is a third equivalent way of expressing separation, by
using the set-difference $K-G \subseteq W$. It is immediate that this is the finest preorder
separating $K$ and $G$, since
\begin{equation*}
G \le_{P} K \iff K-G \ge_{P} 0  \iff K - G \subseteq P.
\end{equation*}
The finest preorder therefore characterizes all separating preorders by the above containment.
Analogously, an oriented map $\psi:W \to (Y,Y_+)$ is separating
precisely when it is monotone in the finest separating preorder $K-G$, since
\begin{equation*}
\psi(G) \leq \psi(K)  \iff \psi(K-G) \geq 0.
\end{equation*}
We characterize the separating subspaces.
\begin{theorem}[Characterization of separating subspaces]
\label{loc:characterization_of_separating_subspaces.statement}
For convex cones $K, G \subseteq W$, a subspace $U \subseteq W$ is a separating subspace of $K$ and $G$ when, equivalently, 
\begin{enumerate}
\item $U$ is a supporting subspace of $K-G$.
\item $U = \ker \psi$ for some oriented map $\psi:W \to (Y, Y_+)$ such that $\psi(G)\le \psi(K)$.
\item $U = \lin(P)$ for a preorder $P \subseteq W$ such that $G \le_{P} K$.
\item $U$ is a supporting subspace of both $K$ and $G$, and such that $\forall k  \in K, g  \in G, k-g \in U \implies k, g \in U$.
\end{enumerate}
\end{theorem}
Point 4 is novel, and notable because it is purely algebraic, without reference to the separating maps from which we derived them. This characterization is of a similar kind to the algebraic, classical definition of the faces of convex cones. It recovers the algebraic definition of faces of $K$ when taking $G := -K$. 
This structure is related to the fact that the supporting subspaces of a convex cone $K$ are equivalently the subspaces separating $K$ from its negation $-K$, a fact that becomes clear when considering point 2: 
\begin{equation*}
\psi(K) \geq 0  \implies \psi(-K) \leq 0  \implies \psi(-K) \leq \psi(K).
\end{equation*}
\begin{proof}[\hypertarget{loc:characterization_of_separating_subspaces.proof}{}Proof of \Cref{loc:characterization_of_separating_subspaces.statement}]

$1 \implies 2$: By the definition of supporting subspaces (\Cref{loc:supporting_subspace_definition_from_oriented_maps.statement}) there is an oriented map $\psi:W \to (Y, Y_+)$ that is positive on
$K-G$ and such that $\ker \psi = U$. Since $K \subseteq K-G$ and $-G \subseteq K-G$, we have that $\psi(K) \ge 0$ and $\psi(G) \le 0$, hence $\psi(G) \leq \psi(K)$.

$2 \implies 3$: The preorder $\psi^{-1}(Y_+)$ satisfies the requirements, since
\begin{equation*}
\ker \psi = \{ u  \in W \mid u \sim_{P} 0\}.
\end{equation*}
in the order $P :=\psi ^{-1}(Y_+)$.

$3 \implies 4$: The subspace $U=\lin(P)$ supports $K$ since $K \geq_{P} 0$ by \Cref{loc:characterization_of_supporting_subspaces.statement} point 4, and supports $G$ by the same argument with the preorder $-P$.
Note that $k-g  \in U  \implies k-g \sim_{P} 0  \implies k \sim_{P} g$. But we also have $k \geq_{P} 0$ and $g \leq_{P} 0$, therefore $k \sim_{P} 0,g \sim_{P} 0$ individually, and therefore $k,g \in U$.

$4  \implies 1$: We show that $U$ is a supporting subspace by showing it satisfies \Cref{loc:characterization_of_supporting_subspaces.statement} point 2. Let $x  \in W, y  \in U$ be such that $0 \le_{K-G} x \le_{K-G} y$; it suffices to show that $x  \in U$. From these inequalities, there exists $k_1, k_2  \in K$ and $g_1, g_2  \in G$ such that $x = k_1 -g_1$ and $y = k_1 + k_2 - g_1-g_2$.  With $k := (k_1 + k_2)\in K$ and $g := (g_1 + g_2)\in G$, we have $k-g =y \in U$, therefore by point 4 we find that $k,g  \in U$. Since $U$ supports $K$ (again by point 4), we have that $U \cap K$ is a face of $K$, so by the algebraic definition of faces (\Cref{loc:characterization_of_the_faces_of_convex_cones.statement} point 1) $k_1, k_2 \in K \cap U$. By a similar argument, $g_1, g_2 \in G \cap U$.
From this, $x = k_1 - g_1 \in U$, which concludes the proof.
\end{proof}
Separating subspaces reveal two faces by intersection with each convex cone.
\begin{proposition}[Decomposition of faces of difference of convex cones]
\label{loc:decomposition_of_faces_of_difference_of_convex_cones.statement}
For two convex cones $K, G \subseteq W$, and any face $F \triangleleft K-G$, the
sets
$F_K := F \cap K$ and $F_G := (-F) \cap G$ are faces of $K$ and $G$ respectively,
and moreover, they are such that 
\begin{equation*}
F = F_K-F_G.
\end{equation*}

Further, for any supporting subspace $U$ of $K-G$ such that $F = U \cap (K-G)$,
we have $F_K = U \cap K$ and $F_G = U \cap G$.
\end{proposition}
\begin{proof}[\hypertarget{loc:decomposition_of_faces_of_difference_of_convex_cones.proof}{}Proof of \Cref{loc:decomposition_of_faces_of_difference_of_convex_cones.statement}]

We first show the second part of the statement, and then the first.
Since $F$ is a face of $K-G$, it has a supporting subspace $U$ of $K-G$ for
which $F =  U \cap (K-G)$, since by \Cref{loc:face_of_convex_cone.statement}, the face is given by intersection with a kernel, and this kernel is a supporting subspace by \Cref{loc:supporting_subspace_definition_from_oriented_maps.statement}. The supporting subspace $U$ is then a separating subspace of $K$
and $G$ by \Cref{loc:characterization_of_separating_subspaces.statement} point 1.
By \Cref{loc:characterization_of_separating_subspaces.statement} point 4, the subspace $U$
supports both $K$ and $G$, so $F_K := K \cap U \triangleleft K$ and $F_G := G \cap U \triangleleft G$
by \Cref{loc:characterization_of_supporting_subspaces.statement} point 5 (we show
$F_K$ and $F_G$ are as in the statement below).
The same point 4 gives $F = (K - G) \cap U = K \cap U - G \cap U$: the containment $\supseteq$ holds
because $U$ is a subspace, and for the inclusion $\subseteq$, any $k - g \in (K-G) \cap U$ has
$k, g \in U$ by \Cref{loc:characterization_of_separating_subspaces.statement} point 4. This
shows the second part of the statement.

For the first part of the statement, it remains only to show that $F_K := K \cap U = F \cap K$,
and that similarly $F_G := G \cap U = (-F) \cap G$.
Recall that $F := (K-G) \cap U$, and note that
$F_K = K \cap U \subseteq (K - G) \cap U = F \subseteq U$, so by taking intersection with $K$ on both sides, we have a sandwich which yields
$F_K = F \cap K$. Similarly, since $G \cap U \subseteq (G-K) \cap U = -F \subseteq U$, we have $F_G = G \cap (-F)$.
\end{proof}
\subsection{Separating flats}
\label{loc:body.separation.separating_flats}
For convex cones $K, G \subseteq X \times \mathbb{R}$, and a preorder $P \subseteq X \times \mathbb{R}$, how does the statement $G \le_{P} K$ manifest in the affine slice $X = \gamma_1^{-1}(X \times \mathbb{R})$?
The three regions of $P$, namely the positive cone $P$, the negative cone $-P$,
and the lineality $\lin(P)$, dehomogenize respectively to a ``positive" convex set
$\gamma_1^{-1}(P)$, a ``negative" convex set $\gamma_1^{-1}(-P)$, and an affine flat
$\gamma_1^{-1}(\lin(P))$. $P$ also induces a preorder on $X$, since $\forall x,y  \in X,$
\begin{equation*}
\gamma_1(y) \leq_{P} \gamma_1(x) \iff \gamma_1(x) -\gamma_1(y)  \in  P  \iff \gamma_0(x-y)  \in  P  \iff  x-y  \in  \gamma_0^{-1}(P),
\end{equation*}
making $\gamma_0^{-1}(P)$ the induced preorder.

The three zones and the preorder can be understood as the affine manifestation of a general convex cone in the homogenization space $X \times \mathbb{R}$.
In \Cref{loc:body.positivity_minimization_and_faces.dehomogenization}, we focused on homogenization cones $K \subseteq (X \times \mathbb{R}_{++}) \cup \{0\}$.
These cones are the ones with trivial negative set $\gamma_1^{-1}(-K) = \varnothing$, trivial neutral flat $\gamma_1^{-1}(\lin(K))= \varnothing$, 
and trivial induced preorder $\gamma_0^{-1}(K) = \{0\}$. As such, they are
suitable to represent individual convex sets, carrying no additional baggage. 
In contrast, a single general convex cone directly induces additional structure which is suggestive of separation: a positive and a negative zone which are ordered in a given preorder,
with their intersection forming an affine flat between the two zones. We think of this structure as separating two convex sets when they are contained in the positive and negative zones respectively.
This motivates separation as a fundamental object of study, since it arises from the dehomogenization of a single general convex cone in the homogenization space.

As before, we express $P$ as in \Cref{loc:preorders_are_characterized_by_oriented_maps.statement} with a linear oriented map $\psi:X \times \mathbb{R} \to (Y, Y_+)$ such that $P = \psi^{-1}(Y_+)$, and then dehomogenize the map by precomposition with $\gamma_1$ as in \Cref{loc:oriented_linear_maps_are_in_bijection_with_oriented_affine_maps.statement}. We derive the separation of two convex sets
$C,  D \subseteq X$ from the conic separation $\psi(G) \le \psi(K)$ of the
homogenization cones $K := \cone \circ\gamma_1(C)$ and $G := \cone \circ \gamma_1(D)$. 
Dehomogenizing  all three objects, we find an affine oriented map $f := \psi \circ \gamma_1$ such that
$f(D) \le 0 \le f(C)$ by \Cref{loc:preservation_of_orientation_in_dehomogenization.statement}, which motivates our next definition.
\begin{definition}[Separating affine maps]
\label{loc:separating_affine_maps.statement}
Given two convex sets $C, D \subseteq X$, an affine oriented map
$f:X \to (Y, Y_+)$ separates $C$ from $D$ when
\begin{equation*}
f(D) \le 0 \le f(C)
\end{equation*}
(or if $f(C) \le 0 \le f(D)$).
\end{definition}
Note that when the codomain of $f$ is $\mathbb{R}$, the convex sets are separated by the hyperplane $\ker f$, and our definition recovers separation by functionals and hyperplanes.
On the other extreme, a separating affine map may have an empty kernel: for two disjoint convex sets $C, D \subseteq X$, it can be shown that the difference $K - G$ of their respective homogenization cones is a pointed convex cone, because both $K$ and $G$ are pointed convex cones with intersection $K \cap G = \{ 0 \}$ (see \hyperlink{loc:generated_faces_are_formed_by_symmetric_differences.proof}{the proof} of \Cref{loc:generated_faces_are_formed_by_symmetric_differences.statement}). Equipping the embedding map with the orientation $K-G$ yields $\tilde{\gamma}_1:X \to (X \times \mathbb{R}, K-G)$, an oriented affine map which is separating because $\tilde{\gamma}_1(D)\le 0 \le \tilde{\gamma}_1(C)$, and which is such that $\ker \tilde{\gamma}_1 = \varnothing$.

One interesting special case of separation is when $D$ is a singleton contained in $C$.
In this case, our notion of separation recovers concepts of locality; geometric objects
that are local at a point.
For $x \in C$ and an affine map $f$ separating $\{x\}$ from $C$, we have $f(x) \leq 0 \leq f(C)$, and $x \in C$ forces $f(x)=0$. 
The separating maps therefore correspond to the oriented linear maps positive on $C$ when $x$ is fixed to be the origin. 
The subset of these oriented maps with codomain
$(\mathbb{R}, \mathbb{R}_{-})$ forms the normal
cone at $x$.
\begin{theorem}[Oriented separation of convex sets]
\label{loc:oriented_separation_of_convex_sets.statement}
For convex sets $C, D \subseteq X$ and an oriented affine map $f:X \to (Y, Y_+)$,
the following are equivalent characterizations of the fact of $f$ separating $C$ from $D$.
\begin{enumerate}
\item $f(D) \le 0 \le f(C)$.
\item The homogenization of $f$, $\psi(x, \lambda) := f(x) - (1-\lambda) f(0)$, is such that $\psi(G) \le \psi(K)$ for the homogenization cones $K := \cone \circ \gamma_1(C)$ and $G := \cone\circ\gamma_1(D)$.
\end{enumerate}

Further, when $C \cap D \neq \varnothing$, the following are also equivalent to the above.
\begin{enumerate}
\setcounter{enumi}{2}
\item The map $f$ is such that the oriented linear map $A(x) := f(x) - f(0)$ is positive on $C-D$, and further $f$ is such that $f(x) = A(x-x_0)$ for any $x_0 \in C \cap D$.
\item The map $f$ is such that its linear part $A(x) := f(x)-f(0)$ is an oriented linear map that, for any choice of $x_0 \in C \cap D$, separates the convex cones $\cone(C - x_0)$ and $\cone(D- x_0)$, and further $f$ is such that $f(x) = A(x-x_0)$.
\end{enumerate}
\end{theorem}
Points 3 and 4 amount to picking any $x_0 \in C \cap D$ as the origin of $X$:
separation of $C$ and $D$ then reduces to positivity of the linear part on $C-D$,
or equivalently to conic separation of $\cone(C - x_0)$ from $\cone(D-x_0)$.
\begin{proof}[\hypertarget{loc:oriented_separation_of_convex_sets.proof}{}Proof of \Cref{loc:oriented_separation_of_convex_sets.statement}]

$1 \iff 2$ Follows from \Cref{loc:preservation_of_orientation_in_dehomogenization.statement} with the affine map $f$ and convex set $C$, and from \Cref{loc:preservation_of_orientation_in_dehomogenization.statement} applied again to the affine map $-f$ and the convex set $D$.

$1 \implies 3$
For the forward direction, positivity on $C-D$ follows from
\begin{equation*}
A(C - D) = A(C) - A(D) = f(C) - f(D) \geq 0.
\end{equation*}
To show that $f(x) = A(x-x_0)$ for any $x_0 \in C \cap D$, by \Cref{loc:oriented_linear_maps_are_in_bijection_with_oriented_affine_maps.statement} we only need that $f(x_0)=0$. For this, note that any $x_0 \in C \cap D$ satisfies $0 \le f(x_0) \le 0$ from membership in $C$ and $D$ respectively, hence $f(x_0) = 0$ by antisymmetry of the codomain order.

$3  \implies 1$
Fix any $x_0 \in C \cap D$. For $c \in C$ we have $c - x_0 \in C - D$, so $f(c) = A(c - x_0) \ge 0$; for $d \in D$ we have $x_0 - d \in C - D$, so $f(d) = A(d - x_0) \le 0$, which shows point 1.

$1 \implies 4$
As in $(1 \implies 3)$, any $x_0 \in C \cap D$ is such that $f(x_0) = 0$, so $f(x) = A(x - x_0)$. Writing $\tilde{C} := C - x_0$ and $\tilde{D} := D - x_0$, we have $A(\tilde{C}) = f(C) \ge 0$ and $A(\tilde{D}) = f(D) \le 0$. Since $A$ is linear, this positivity extends to $\cone(\tilde{C})$ and the negativity to $\cone(\tilde{D})$, so $A$ separates the two cones, which shows point 4.

$4 \implies 1$
Fix any $x_0 \in C \cap D$. From $A$ separating the cones we have $A(\cone(\tilde{C})) \ge 0$ and $A(\cone(\tilde{D})) \le 0$, in particular $A(\tilde{C}) \ge 0$ and $A(\tilde{D}) \le 0$. Since $f(x) = A(x - x_0)$, this reads $f(C) \ge 0$ and $f(D) \le 0$, which shows point 1.
\end{proof}
We call the kernels of separating affine maps \emph{separating flats}. We prove a number of characterizations for these flats.
\begin{theorem}[Characterization of the separating flats]
\label{loc:characterization_of_the_separating_flats.statement}
Given convex sets $C, D \subseteq X$, an affine flat $L \subseteq X$ is a separating flat when,
equivalently,
\begin{enumerate}
\item $L = \gamma_1^{-1}(U)$ for $U$ a separating subspace of $K = \cone \circ \gamma_1(C)$ and $G = \cone \circ \gamma_1(D)$.
\item $L = \ker f$ for some affine oriented $f:X \to (Y, Y_+)$ separating $C$ from $D$.
\end{enumerate}

And if $C \cap D \neq \varnothing$, the following statements are also equivalent.

\begin{enumerate}
\setcounter{enumi}{2}
\item $L$ is supporting for both $C$ and $D$, and $\forall c  \in C, d  \in D, c-d \in L-L \implies c,d \in L$.
\item $L =  \tilde{L} + C \cap D$ for $\tilde{L}$ a supporting subspace of $\cone(C-D)$.
\end{enumerate}
\end{theorem}
In point 3 we present a surprisingly compact characterization of
separating flats, which is purely algebraic and reminiscent of the
splitting property for faces.
Point 4 shows that separating flats are the same objects as the supporting
subspaces considered in \cite{scottBilateralFacialReduction2026}.
\begin{proof}[\hypertarget{loc:characterization_of_the_separating_flats.proof}{}Proof of \Cref{loc:characterization_of_the_separating_flats.statement}]

$1 \iff 2$: By \Cref{loc:oriented_separation_of_convex_sets.statement}, the affine map $f$ separates $C,D$ if and only if its homogenization $\psi$ separates the corresponding homogenization cones. Then $\ker \psi$ is a separating subspace by \Cref{loc:characterization_of_separating_subspaces.statement} point 2, and $\ker f = \gamma_1^{-1}(\ker \psi)$ by \Cref{loc:dehomogenization_of_kernels.statement}.

$2 \implies 3$: $L$ supports both $C$ and $D$ by \Cref{loc:characterization_of_supporting_flats.statement} point 2. For the separation condition: $\forall c  \in C, d  \in D, c-d \in \ker f - \ker f = \ker A$, (for $A(x) = f(x)-f(0)$ the linear part of $f$).
Then $A(c-d) = 0$, hence $A(c) = A(d)$, and therefore $f(c)=f(d)$.
But also, by their membership in $C,D$ we have $0 \leq f(c) = f(d) \leq 0$, hence $f(c)=f(d)=0$ since the codomain of $f$ is a partial order. This shows point 3.

$3 \implies 2$:
Fix $x_0 \in C \cap D$; it lies in $L$ by taking $c = d = x_0$ in the absorption condition. With $\tilde{L} := L - x_0$, $\tilde{C} := C - x_0$, and $\tilde{D} := D - x_0$, the absorption condition reads $\tilde{c} - \tilde{d} \in \tilde{L} \implies \tilde{c}, \tilde{d} \in \tilde{L}$. By \Cref{loc:characterization_of_separating_subspaces.statement} point 4, $\tilde{L}$ is a separating subspace of $\cone(\tilde{C})$ and $\cone(\tilde{D})$; by its point 2 there is an oriented linear map $A:X \to (Y, Y_+)$ with $\ker A = \tilde{L}$ and $A(\cone \tilde{D}) \le 0 \le A(\cone \tilde{C})$. Then $f(x) := A(x - x_0)$ satisfies $f(D) \le 0 \le f(C)$ with $\ker f = x_0 + \tilde{L} = L$, which is point 2.

$2 \implies 4$:
Let $L = \ker f$ for a separating $f$. By \Cref{loc:oriented_separation_of_convex_sets.statement} point 3, for any $x_0 \in C \cap D$ we have $f(x) = A(x - x_0)$ with $A(x) := f(x) - f(0)$ oriented and positive on $C - D$. Positivity on $C - D$ is positivity on $\cone(C-D)$, so by \Cref{loc:characterization_of_supporting_subspaces.statement} point 3, $\tilde{L} := \ker A$ is a supporting subspace of $\cone(C-D)$. Since $\ker f = x_0 + \ker A$ and $C \cap D \subseteq \ker f$, we get $L = x_0 + \tilde{L} = \tilde{L} + C \cap D$, which is point 4.

$4 \implies 2$:
Let $L = \tilde{L} + C \cap D$ with $\tilde{L}$ a supporting subspace of $\cone(C-D)$. By \Cref{loc:characterization_of_supporting_subspaces.statement} point 3, there is an oriented linear map $A$ with $\ker A = \tilde{L}$ that is positive on $\cone(C-D) \supseteq C - D$. Fix $x_0 \in C \cap D$ and set $f(x) := A(x - x_0)$. By \Cref{loc:oriented_separation_of_convex_sets.statement} point 3, $f$ separates $C$ from $D$, and $\ker f = x_0 + \tilde{L} = L$, from the fact that $C \cap D \subseteq L$. This is point 2.
\end{proof}
\section{Lattice minima}
\label{loc:body.lattice_minima}
In the previous section, it was made apparent that the orders $K - G$ and $\cone(C-D)$ are of canonical importance: they are the minimum separating orders, and characterize all other separating orders. In this section we identify the remaining canonical minima: we show that the \emph{cross-lineality} $\lin(K-G)$ is the lattice minimum of the separating subspaces, and that its dehomogenization $\gamma_1^{-1}(\lin(K-G))$ is the minimum separating flat, and is moreover 
the joint supporting subspace recently introduced in~\cite{scottBilateralFacialReduction2026}. 

We also consider the generated face $F_K(G)$ and the generated supporting
subspace $H_K(G)$, which are respectively the 
minimum face and supporting subspace (in the containment order) of $K$
that contain $K \cap G$. Unlike typical definitions of generated faces~\cite{weisFaceGeneratedPoint2021}, we do not require that $G$
be contained in $K$. From this relaxation, we 
obtain the pairs $(F_K(G), F_G(K))$ and $(H_K(G), H_G(K))$ of mutually
minimal faces and supporting subspaces.
In the next section we show how the two pairs 
and the cross-lineality mutually characterize each other. We do this by first characterizing
the generated faces.
\subsection{Generated faces of convex cones}
\label{loc:body.lattice_minima.generated_faces_of_convex_cones}
\begin{theorem}[Characterization of the generated face of a convex cone]
\label{loc:characterization_of_the_generated_face_of_a_convex_cone.statement}
For convex cones $K, G \subseteq W$, the set $F_K(G)$, the face of $K$ generated by $G$, is equivalently
\begin{enumerate}
\item The face of $K$ that contains $K \cap G$ and that is contained in every other such face.
\item $\min_{K-G} K$
\item $\min_{K - (K \cap G)}K$
\item $\{k \in K \mid \exists w \in G \text{ s.t. } k \leq_{K} w\}$.
\end{enumerate}
\end{theorem}
\begin{proof}[\hypertarget{loc:characterization_of_the_generated_face_of_a_convex_cone.proof}{}Proof of \Cref{loc:characterization_of_the_generated_face_of_a_convex_cone.statement}]

First note that for any preorder $P \subseteq W$ for which $K \geq_{P} 0$, we have 
\begin{equation*}
\min_P K = \{x \in K \mid x \le_{P} K\} = \{x \in K \mid x \le_{P} 0\}.
\end{equation*}
$2 = 4$: With $P := K-G$, from the r.h.s. it suffices to characterize for which $x \in K$ the inequality $x \le_{K-G} 0$ holds. Note that this inequality is equivalently written as
\begin{equation*}
\exists w \in G: x \le_{K} w,
\end{equation*}
which is precisely the selecting property of point 4.

$4 = 3$: Point 4 is equivalently stated as $\{k \in K \mid \exists w \in G \cap K \text{ s.t. } k \leq_{K} w\}$, because $0 \le_{K} k \leq_{K} w  \implies 0 \leq_{K} w  \implies w \in K$. Then obtaining point 4 from point 3 follows from the same argument as in $(2 = 4)$, only with the cone $G \cap K$ instead of $G$.

$2 = 3 = 1$: The set of point 2 is a face of $K$ because it is a minimum set by \Cref{loc:characterization_of_the_faces_of_convex_cones.statement} point 3. It contains $K \cap G$ because $\forall x \in K \cap G, x  \in G  \implies x \geq_{G} 0  \implies x \leq_{K-G} 0$. Finally, it is the smallest such face: consider any $F \triangleleft K$ that contains $K \cap G$. The face $F$ is given by the minimization of some preorder $P \subseteq W$ by \Cref{loc:characterization_of_the_faces_of_convex_cones.statement} point 3. Then $K \cap G \subseteq F  \implies K \cap G \leq_{P} 0$, and with the fact that $K \geq_{P} 0$, we find that $K - (K \cap G) \geq_{P} 0$, i.e., $K - (K \cap G) \subseteq P$. Then since the minimum set is monotone with the underlying preorder,  by considering the generated face as in point 3 we have that
\begin{equation*}
\min_{K - (K \cap G)} K \subseteq \min_P K = F,
\end{equation*}
showing that $\min_{K - (K \cap G)} K$ is indeed the minimum face containing $K \cap G$.
\end{proof}
\begin{remark}
\label{loc:characterization_of_the_generated_face_of_a_convex_cone.remark}
For any preorder $P \supseteq K$, the face $\min_P K$ admits various other expressions. The condition $x \le_{P} 0$ on a point $x \in K$ is equivalent to $x \sim_{P} 0$, so
\begin{equation*}
\min{}_P K = K \cap (-P) = K \cap \lin(P).
\end{equation*}
For $P = K-G$, this reads 
\begin{equation*}
F_K(G) = K \cap (G-K) = K \cap \lin(K-G),
\end{equation*}
 and for $P = K - (K\cap G)$ this reads 
\begin{equation*}
F_K(G) = K \cap ((K \cap G)-K) = K \cap \lin(K- (K \cap G)).
\end{equation*}
\end{remark}
The generated face $F_K(\cdot)$ as a set-valued map has fixed points which are precisely the faces of $K$. Writing out this fact as equations provides novel characterizations for the faces of convex cones.
\begin{corollary}[Faces are fixed points of the generated face]
\label{loc:faces_are_fixed_points_of_the_generated_face.statement}
For a convex cone $K \subseteq W$, a convex cone $S \subseteq W$ is a face of $K$ if and only if
\begin{equation*}
S = \min_{K-S}K, \quad \text{equivalently} \quad S = K \cap (S-K), \quad\text{or}\quad S = K \cap \lin(K-S).
\end{equation*}
\end{corollary}
\begin{proof}[\hypertarget{loc:faces_are_fixed_points_of_the_generated_face.proof}{}Proof of \Cref{loc:faces_are_fixed_points_of_the_generated_face.statement}]

The first formula follows from equivalence of \Cref{loc:characterization_of_the_generated_face_of_a_convex_cone.statement} points 1 and 2, and the two other formulas follow by reformulating the r.h.s. as in \Cref{loc:characterization_of_the_generated_face_of_a_convex_cone.remark}. 
\end{proof}
\subsection{The cross-lineality}
\label{loc:body.lattice_minima.the_cross:lineality}
We show that the cross-lineality $\lin(K-G)$ is the lattice minimum of the separating
subspaces, and that it admits two decompositions: that it decomposes into the two
mutually generated faces of $K$ and $G$, and that it factors with the intersection $K \cap G$.
We first list some properties of the lineality for future reference,
leaving the proof to the appendix.
\begin{lemma}[Characterization of the lineality]
\label{loc:characterization_of_the_lineality.statement}
Given a convex cone $K \subseteq W$, the lineality $\lin(K)$ is, equivalently,
\begin{enumerate}
\item $K \cap -K$
\item $\{x \sim_{K} 0 \mid x \in K\}.$
\item The largest subspace contained in $K$; it contains every subspace contained in $K$. 
\item The minimum supporting subspace of $K$.
\item The minimum face of $K$.
\end{enumerate}
\end{lemma}
\begin{lemma}[The cross-lineality is the minimum separating subspace]
\label{loc:the_cross:lineality_is_the_minimum_separating_subspace.statement}
For convex cones $K, G \subseteq W$, the subspace $\lin(K-G)$ is the minimum separating subspace of $K$ and $G$.
\end{lemma}
\begin{proof}[\hypertarget{loc:the_cross:lineality_is_the_minimum_separating_subspace.proof}{}Proof of \Cref{loc:the_cross:lineality_is_the_minimum_separating_subspace.statement}]

Considering \Cref{loc:characterization_of_the_lineality.statement} point 4, $\lin(K-G)$ is the smallest
supporting subspace of $K-G$, and then by 
\Cref{loc:characterization_of_separating_subspaces.statement} point 1, it is the smallest separating subspace of $K$ and $G$.
\end{proof}
\begin{lemma}[Cross-lineality from the difference of generated faces]
\label{loc:cross:lineality_from_the_difference_of_generated_faces.statement}
For convex cones $K, G \subseteq W$,
\begin{equation*}
\lin(K-G)= F_K(G)-F_G(K).
\end{equation*}
\end{lemma}
\begin{proof}[\hypertarget{loc:cross:lineality_from_the_difference_of_generated_faces.proof}{}Proof of \Cref{loc:cross:lineality_from_the_difference_of_generated_faces.statement}]

By \Cref{loc:characterization_of_the_lineality.statement} points 4 and 5, the subspace $\lin(K-G)$ both supports $K-G$ and is its minimum face. Applying \Cref{loc:decomposition_of_faces_of_difference_of_convex_cones.statement} with $F = \lin(K-G)$, so that $F_K = \lin(K-G) \cap K$ and $F_G = \lin(K-G) \cap G$, we find that
\begin{equation*}
\lin(K-G) = \bigl(\lin(K-G) \cap K\bigr) - \bigl(\lin(K-G) \cap G\bigr) = F_K(G) - F_G(K),
\end{equation*}
where the second equality follows by \Cref{loc:characterization_of_the_generated_face_of_a_convex_cone.remark}.
\end{proof}
The next two results detail the case where one convex cone is contained in the other, in which case the cross-lineality recovers the generated supporting subspace $H_K(G)$.
\begin{proposition}[Separation reduces to containment]
\label{loc:separation_reduces_to_containment.statement}
Let $K \subseteq W$ be a convex cone and let the convex cone $S \subseteq W$ be either a subcone of $K$ or a subspace of $W$.
\begin{enumerate}
\item A subspace $U \subseteq W$ is a separating subspace of $K$ and $S$ if and only if $U$ is a supporting subspace of $K$ with $S \subseteq U$.
\item $\lin(K-S)$ is the minimum supporting subspace of $K$ containing $S$.
\end{enumerate}
\end{proposition}
\begin{proof}[\hypertarget{loc:separation_reduces_to_containment.proof}{}Proof of \Cref{loc:separation_reduces_to_containment.statement}]

Point 1. $(\Rightarrow)$ By \Cref{loc:characterization_of_separating_subspaces.statement} point 4, $U$ supports $K$. By point 3 of that statement, $U = \{u \in W \mid u \sim_{P} 0\}$ for a preorder $P$ such that $S \le_{P} K$. Then $0  \in K  \implies S \le_{P} 0$. Then the reverse inequality $S \ge_{P} 0$ also holds, because either $S \subseteq K$, or $S$ is a subspace in which case $S = -S$, so the first inequality implies the reverse. Hence $S \subseteq U$.
$(\Leftarrow)$ By \Cref{loc:supporting_subspace_definition_from_oriented_maps.statement}, $U = \ker \psi$ for an oriented map $\psi$ positive on $K$. Then $\psi(S) = 0 \le \psi(K)$, so $U$ is separating by \Cref{loc:characterization_of_separating_subspaces.statement} point 2.

Point 2. \Cref{loc:the_cross:lineality_is_the_minimum_separating_subspace.statement} states for the pair $(K, S)$ that the subspace $\lin(K-S)$ is the minimum separating subspace of $K$ and $S$. Point 1 shown above then identifies $\lin(K-S)$ as the minimum supporting subspace of $K$ containing $S$.
\end{proof}
\begin{corollary}[The generated supporting subspace is a cross-lineality]
\label{loc:the_generated_supporting_subspace_is_a_cross:lineality.statement}
Let $K, G \subseteq W$ be convex cones. The generated supporting
subspace is given by 
\begin{equation*}
H_K(G)=  \lin(K- (K \cap G)).
\end{equation*}
\end{corollary}
\begin{proof}[\hypertarget{loc:the_generated_supporting_subspace_is_a_cross:lineality.proof}{}Proof of \Cref{loc:the_generated_supporting_subspace_is_a_cross:lineality.statement}]

Follows from \Cref{loc:separation_reduces_to_containment.statement} point 2 with $S := K \cap G$.
\end{proof}
\begin{lemma}[Cross-lineality with containment]
\label{loc:cross:lineality_with_containment.statement}
Given convex cones $S \subseteq K \subseteq W$,
\begin{equation*}
\lin(K - S) = F_K(S) - S.
\end{equation*}
\end{lemma}
\begin{proof}[\hypertarget{loc:cross:lineality_with_containment.proof}{}Proof of \Cref{loc:cross:lineality_with_containment.statement}]

Applying \Cref{loc:cross:lineality_from_the_difference_of_generated_faces.statement} to the pair $(K, S)$ gives
\begin{equation*}
\lin\bigl(K-S\bigr) = F_K(S) - F_{S}(K) = F_K(S) - S,
\end{equation*}
since $F_S(K) = S$ because $S \subseteq K$ and $S$ is a face of itself.
\end{proof}
\begin{theorem}[Characterization of the cross-lineality]
\label{loc:cross_lineality.statement}
For two convex cones $K, G \subseteq W$, the \emph{cross-lineality} of the pair $(K, G)$ is equivalently defined as
\begin{enumerate}
\item $(K-G) \cap (G -K)=\lin(K-G)=\lin(G-K).$
\item The minimum separating subspace of $K$ and $G$.
\item $H_K(G)+ H_G(K)$
\item $F_K(G) - F_G(K)$
\end{enumerate}
\end{theorem}
\begin{proof}[\hypertarget{loc:cross_lineality.proof}{}Proof of \Cref{loc:cross_lineality.statement}]

$1 = 2$: \Cref{loc:the_cross:lineality_is_the_minimum_separating_subspace.statement}.

$1 = 4$: \Cref{loc:cross:lineality_from_the_difference_of_generated_faces.statement}

$4 = 3$: 
By applying \Cref{loc:cross:lineality_with_containment.statement} twice, we have
\begin{align*}
&\lin(K - (K \cap G)) + \lin(G - (K \cap G))\\
=&\lin(K - (K \cap G)) - \lin(G - (K \cap G))
= F_K(G) - K \cap G - F_G(K) + K \cap G,
\end{align*}

and the r.h.s. equals point 4 because $K \cap G$ lies in both generated faces.
\end{proof}
Finally, we show the relation between the generated face and supporting subspace.
\begin{proposition}[Generated face spans the lineality]
\label{loc:generated_face_spans_the_lineality.statement}
For convex cones $K, G \subseteq W$,
\begin{equation*}
H_K(G) = \Span(F_K(G)).
\end{equation*}
\end{proposition}
\begin{proof}[\hypertarget{loc:generated_face_spans_the_lineality.proof}{}Proof of \Cref{loc:generated_face_spans_the_lineality.statement}]

We show the following equality chain:
\begin{align*}
H_K(G) = &\lin(K - (K \cap G))\\
=&F_K(G) - (K\cap G) 
= \Span(F_K(G)).
\end{align*}
The first equality is \Cref{loc:the_generated_supporting_subspace_is_a_cross:lineality.statement}.
The second equality follows from \Cref{loc:cross:lineality_with_containment.statement} with $S := K \cap G$.

We show the last equality.
$(\subseteq)$: Follows because $K\cap G \subseteq F_K(G)$, which is immediate from \Cref{loc:characterization_of_the_generated_face_of_a_convex_cone.statement} point 4.
$(\supseteq)$: The left-hand side is a subspace by the equality with $\lin(K-(K\cap G))$. It contains both $K\cap G$ and $F_K(G)$ because $0$ is in both $K\cap G$ and $F_K(G)$, therefore it contains the span also.
\end{proof}
In this section we have shown that the generated faces, the generated supporting subspaces, and the cross-lineality determine one another. We summarize our findings in the following diagram.

\begin{center}
\tikzcdset{every label/.append style={font=\small}}
\begin{tikzcd}[row sep=8ex, column sep=3em]
  & \mathrm{lin}(K-G)
      \arrow[dl, shift left=1ex, "K \cap \mathrm{lin}(K-G)"] & \\
  {F_K(G),\; F_G(K)}
      \arrow[ur, shift left=1ex, "F_K(G)-F_G(K)"]
      \arrow[rr, shift left=1ex, "\mathrm{Span}\,F_K(G)"]
  & & {\begin{array}{c} H_K(G) = \mathrm{lin}(K - K \cap G) \\ H_G(K) = \mathrm{lin}(G - K \cap G) \end{array}}
      \arrow[ul, "H_K(G)+H_G(K)"']
      \arrow[ll, shift left=1ex, "K \cap H_K(G)"]
\end{tikzcd}
\end{center}

\subsection{Generated faces of convex sets}
\label{loc:body.lattice_minima.generated_faces_of_convex_sets}
We dehomogenize the generated faces of convex cones and the cross-lineality to obtain the
analogues for convex sets. We obtain the generated face $F_C(D)$: the smallest face of $C$
containing $C \cap D$ for arbitrary convex sets $C,D \subseteq X$.
This generalizes the generated face at a point~\cite{weisFaceGeneratedPoint2021},
which is recovered exactly by requiring that the argument $D$ be a singleton contained in $C$.
\begin{theorem}[Characterization of the generated face of a convex set]
\label{loc:characterization_of_the_generated_face_of_a_convex_set.statement}
For convex sets $C, D \subseteq X$ with homogenization cones $K, G \subseteq X \times \mathbb{R}$, the generated face 
$F_C(D)$ is equivalently:
\begin{enumerate}
\item The smallest face of $C$ that contains $C \cap D$.
\item $\gamma_1^{-1}(F_K(G))$.
\item $\min_{\cone(C-D)}C$ when $C \cap D \neq \varnothing$, otherwise it is $\varnothing$.
\item $\min_{\cone(C- (C \cap D))}C$ when $C \cap D \neq \varnothing$, otherwise it is $\varnothing$.
\end{enumerate}
\end{theorem}
\begin{remark}
\label{loc:characterization_of_the_generated_face_of_a_convex_set.remark}
For a preorder $P \subseteq X$ with $C \cap D \subseteq \min_P C$, we reformulate $\min_P C$ in different ways. For $x \in C$ and $x_0 \in C \cap D$, we have $x_0  \in \min_P C$, and therefore minimality of $x$ is equivalent to either $x \le_{P} x_0$ or $x \sim_{P} x_0$. Hence
\begin{equation*}
\min{}_P C = C \cap (C \cap D - P) = C \cap (\lin(P) + C \cap D).
\end{equation*}
Letting $P := \cone(C-D)$, this reads 
\begin{equation*}
F_C(D) = C \cap [\cone(D-C) + C \cap D] = C \cap (\lin\cone(C-D) + C \cap D),
\end{equation*}
and with $P = \cone(C - (C \cap D))$,
\begin{equation*}
F_C(D) = C \cap [\cone((C \cap D)-C) + C \cap D] = C \cap (\lin\cone(C- (C \cap D)) + C \cap D).
\end{equation*}
\end{remark}
For a singleton $D = \{x\}$ contained in $C$, from the above remark we have
\begin{equation*}
F_C(\{x\}) = C \cap (\lin \cone(C - x) + x),
\end{equation*}
recovering \cite[Proposition 2.4]{millanIntrinsicCoreMinimal2023}. Our proof method provides an illuminating perspective for this formula, in that the expression
$C \cap (\lin\cone(C-x) + x)$
is a minimum set in the preorder $P := \cone(C - x)$.

We show \Cref{loc:characterization_of_the_generated_face_of_a_convex_set.statement} by dehomogenizing the analogous statements in \Cref{loc:characterization_of_the_generated_face_of_a_convex_cone.statement}. To do this, we require additional lemmas building on \Cref{loc:body.positivity_minimization_and_faces.dehomogenization} describing the dehomogenization of faces subject to a containment condition, dehomogenization of preorders, and dehomogenization of minimum sets.
The first lemma restricts the bijection of faces of \Cref{loc:dehomogenization_is_a_bijection_on_the_faces.statement} with a containment requirement.
\begin{lemma}[Bijection of faces containing a subset]
\label{loc:bijection_of_faces_containing_a_subset.statement}
Let $S \subseteq C \subseteq X$ be convex sets with homogenization cones $K = \cone \circ \gamma_1(C)$ and $T = \cone \circ \gamma_1(S)$. Then the set-valued map $\gamma_1^{-1}(\cdot)$ restricts to a lattice isomorphism between the faces of $C$ that contain $S$ and the faces of $K$ that contain $T$.
\end{lemma}
We defer \hyperlink{loc:bijection_of_faces_containing_a_subset.proof}{the proof} to the appendix.
Recall that a preorder $P \subseteq X  \times \mathbb{R}$ induces the preorder $\gamma_0^{-1}(P)$ on $X$. We find an expression for the dehomogenization of the finest preorder $P := K-G$.
\begin{lemma}[Differences in the lifting]
\label{loc:differences_in_the_lifting.statement}
For convex sets $C, D \subseteq X$, let $K = \cone(C \times \{1\})$ and
$G = \cone(D \times \{1\})$. Then
\begin{equation*}
\gamma_0^{-1}(K - G) = \cone(C-D).
\end{equation*}
\end{lemma}
We defer \hyperlink{loc:differences_in_the_lifting.proof}{the proof} to the appendix.
Next, we show sets of minima also dehomogenize, by dehomogenizing jointly the feasible set and the preorder under which the minimum is taken.
\begin{lemma}[Minima dehomogenize to the induced order]
\label{loc:minima_dehomogenize_to_the_induced_order.statement}
Let $K \subseteq X \times \mathbb{R}$ be a convex cone with slice $C := \gamma_1^{-1}(K)$, and let $P \subseteq X \times \mathbb{R}$ be a preorder. Then
\begin{equation*}
\gamma_1^{-1}(\min{}_P K) \subseteq \min{}_{\gamma_0^{-1}(P)}C,
\end{equation*}
with equality when the left hand side is non-empty.
\end{lemma}
We defer \hyperlink{loc:minima_dehomogenize_to_the_induced_order.proof}{the proof} to the appendix.
The proof of \Cref{loc:characterization_of_the_generated_face_of_a_convex_set.statement} that
follows comprises only dehomogenization arguments applied to the corresponding
points of \Cref{loc:characterization_of_the_generated_face_of_a_convex_cone.statement}.
\begin{proof}[\hypertarget{loc:characterization_of_the_generated_face_of_a_convex_set.proof}{}Proof of \Cref{loc:characterization_of_the_generated_face_of_a_convex_set.statement}]

The subcone $K \cap G \subseteq K$ has slice $C \cap D$.

$1 = 2$:
By \Cref{loc:bijection_of_faces_containing_a_subset.statement}, the minimum face of $K$ containing $K \cap G$, which is $F_K(G)$ by \Cref{loc:characterization_of_the_generated_face_of_a_convex_cone.statement} point 1, has as its slice the smallest face of $C$ containing $C \cap D$, which is point 1.

For the equality of points 2, 3, and 4: if $C \cap D = \varnothing$, then all three points are $\varnothing$.
Below we only consider the case where $C \cap D \neq \varnothing$.

$2 = 3$:
For $x_0 \in C \cap D$, the lift $\gamma_1(x_0)$ lies in $K \cap G \subseteq F_K(G) = \min_{K-G}K$ by \Cref{loc:characterization_of_the_generated_face_of_a_convex_cone.statement} points 1 and 2. Then the minimum has a non-empty slice, and \Cref{loc:minima_dehomogenize_to_the_induced_order.statement} applies. The induced order is $\gamma_0^{-1}(K-G) = \cone(C-D)$ by \Cref{loc:differences_in_the_lifting.statement}, whence $\gamma_1^{-1}(F_K(G)) = \min_{\cone(C-D)}C$.

$2 = 4$:
Follows from a similar argument as $(2 = 3)$, for the convex sets $C$ and $C \cap D$, whose homogenization cones are $K$ and $K \cap G$ respectively (by \Cref{loc:dehomogenization_is_a_lattice_isomorphism_between_convex_sets_and_homogenization_cones.statement}). This argument yields 
\begin{equation*}
\gamma_1^{-1}(F_K(K \cap G)) = \min_{\cone(C - C \cap D)} C,
\end{equation*}
and the r.h.s. corresponds to $F_K(G)$ by the equivalence of \Cref{loc:characterization_of_the_generated_face_of_a_convex_cone.statement} points 2 and 3.
\end{proof}
\begin{remark}
\label{loc:characterization_of_the_generated_face_of_a_convex_set.monotonicity_remark}
That $\min_{\cone(C - (C \cap D))}C$ is the minimum face containing $C \cap D$ also admits a direct proof. It can be argued that $\cone(C - (C \cap D))$ is the finest preorder for which $C \cap D \leq C$. Indeed, for any preorder $P \subseteq X$ such that $C \cap D \leq C$, it follows that $P \supseteq C - (C \cap D)$, so $\cone(C - (C \cap D))$ is the least such preorder. Enlarging the preorder adds comparisons, hence it enlarges the minimum set: $P \subseteq Q$ gives $\min_P C \subseteq \min_Q C$. The least preorder therefore produces the least face.
\end{remark}
\subsection{The joint supporting subspace}
\label{loc:body.lattice_minima.the_joint_supporting_subspace}
In this section we show that the joint supporting subspace~\cite{scottBilateralFacialReduction2026}
is the dehomogenization of the cross-lineality, and that
the characterizations of the cross-lineality in \Cref{loc:cross_lineality.statement}
dehomogenize to analogues for the joint supporting subspace.
We thereby generalize characterizations from~\cite{scottBilateralFacialReduction2026}
to general vector spaces. In the next theorem we use the notation $\{ a,b \mid \ldots \}$ to mean the set including both elements $a$ and $b$ whenever the condition is satisfied.
\begin{theorem}[Characterizations of the joint supporting subspace]
\label{loc:affine_cross:lineality.statement}
Let $C, D \subseteq X$ be convex sets,
with their homogenization cones $K = \cone(C \times \{ 1 \})$ and $G = \cone(D \times \{ 1 \})$.
The \emph{joint supporting subspace} $T_a(C, D)$ is equivalently:
\begin{enumerate}
\item $\gamma_1^{-1}(\lin(K-G))$
\item $\lin\cone(C-D) + C \cap D$
\item The minimum separating flat of $C$ and $D$.
\item $\lin \cone(C - (C \cap D)) + \lin \cone(D - (C \cap D)) + C \cap D$.
\item $\aff(F_C(D) \cup F_D(C))$
\item $\aff\{ c,c',d,d'\mid c, c'  \in C , d,d'  \in D \text{ s.t. } c-d = \lambda(d'-c') \text{ for some } \lambda > 0\}$
\end{enumerate}
\end{theorem}
Characterizations 1, 3 and 6 are novel even in finite dimensions, and characterizations 2, 4 and 5 extend
known characterizations in finite dimensions~\cite{scottBilateralFacialReduction2026} to general vector spaces.
Characterization 1 relates to the cross-lineality and characterization 2 to the finest separating order $\cone(C-D)$.
In characterization 2 we also recover the definition of the JSS presented in~\cite{scottBilateralFacialReduction2026},
showing that our characterizations are consistent with this earlier work.
Characterizations 4 and 5 are decompositions, point 5 with the generated faces and point 4 with the generated supporting flats, as we show below.

With characterization 5, we find that the two generated faces $F_C(D),  F_D(C)$ and the joint supporting subspace $T_a(C, D)$ determine one another, with the converse given by \Cref{loc:characterization_of_the_generated_face_of_a_convex_set.remark}:
\begin{equation*}
F_C(D) =  T_a(C, D) \cap C,  \quad T_a(C, D) =  \aff(F_C(D) \cup  F_D(C)).
\end{equation*}

This relationship clarifies the link between the joint supporting subspace in \cite{scottBilateralFacialReduction2026} and the pair of ``minimal" faces in \cite{linFacialReductionNice2025}, also showing that the minimal faces are in fact the mutually generated faces.

Characterization 6 of \Cref{loc:affine_cross:lineality.statement} is particularly direct: it says that the joint supporting subspace
is the affine hull of those points in $C$ and $D$ that form symmetric differences. This is the affine counterpart of the structure of the cross-lineality $\lin(K-G)= (K-G) \cap (G-K)$, which is the set of points $k,k' \in K$, $g,g' \in G$ that form symmetric differences $k-g=g'-k'$.
The relationship of point 6 to point 2 is also interesting: a symmetric difference $c-d = d'-c'$ immediately yields the point $p = 1/2(c + c') = 1/2(d + d') \in C \cap D$; the intersection arises immediately and necessarily, and is sufficient to specify the affine offset of the joint supporting subspace.
The following proposition serves to prove characterization 6, but also strengthens it.
\begin{proposition}[Pairs forming symmetric differences]
\label{loc:pairs_forming_symmetric_differences.statement}
Let $C, D \subseteq X$ be convex sets and let $K = \cone(C \times \{1\})$, $G = \cone(D \times \{1\})$ be their homogenization cones. For $c \in C$ and $d \in D$, the following are equivalent:
\begin{enumerate}
\item The difference $c-d$ has an opposite difference up to scale: $c - d = -\lambda(c'-d')$ for some $\lambda > 0$, $c' \in C$, $d' \in D$.
\item $c - d \in \lin\cone(C-D)$.
\item $c, d \in \gamma_1^{-1}(\lin(K-G))$.
\end{enumerate}
\end{proposition}
With characterization 6 of \Cref{loc:affine_cross:lineality.statement}, we find that the points forming symmetric differences naturally assemble in an affine flat: their affine hull contains exactly those points that form symmetric differences, and no other.
The joint supporting subspace is therefore ``the affine structure of the points forming symmetric differences between the two convex sets".
This emergent geometric regularity is explained by the underlying cross-lineality $\lin(K-G)$, that is similarly the set of points forming symmetric differences $k-g=g'-k'$. That 
$\lin(K-G)$ is a subspace emerges from the convexity of $K$ and $G$.
\begin{proof}[\hypertarget{loc:pairs_forming_symmetric_differences.proof}{}Proof of \Cref{loc:pairs_forming_symmetric_differences.statement}]

$1 \iff 2$: If $c = d$, both hold: point 2 because $\lin\cone(C-D)$ contains $0$, and point 1 with $c' = d' = c \in C \cap D$. Otherwise no witnessing $\lambda$ can vanish, so point 1 is equivalently $c-d \in \cone(D-C)$. Since $c-d$ always lies in $C - D \subseteq \cone(C-D)$, point 1 is equivalent to
\begin{equation*}
c-d \in \cone(C-D) \cap \cone(D-C) = \lin\cone(C-D).
\end{equation*}
$2 \iff 3$: By \Cref{loc:differences_in_the_lifting.statement} applied to both $(C,D)$ and $(D,C)$,
\begin{equation*}
\gamma_0^{-1}(\lin(K-G)) = \gamma_0^{-1}(K-G) \cap \gamma_0^{-1}(G-K) = \cone(C-D) \cap \cone(D-C) = \lin\cone(C-D),
\end{equation*}
so point 2 is equivalently $\gamma_1(c) - \gamma_1(d) \in \lin(K-G)$. Since $\lin(K-G)$ is a separating subspace of $K$ and $G$ by \Cref{loc:cross_lineality.statement} point 2, absorption (\Cref{loc:characterization_of_separating_subspaces.statement} point 4) places $\gamma_1(c), \gamma_1(d) \in \lin(K-G)$ individually, which is point 3. Conversely, if $\gamma_1(c), \gamma_1(d) \in \lin(K-G)$, then
$\gamma_1(c) - \gamma_1(d)  \in \lin(K-G)$, hence $c-d  \in \gamma_0^{-1}(\lin(K-G)) = \lin \cone(C-D)$,
which is point 2.
\end{proof}
The conditions on $c$ and on $d$ in point 3 of \Cref{loc:pairs_forming_symmetric_differences.statement} are independent. We can therefore consider only one convex set, and doing so characterizes the generated faces, because $F_C(D)= T_a \cap C$.
\begin{proposition}[Generated faces are formed by symmetric differences]
\label{loc:generated_faces_are_formed_by_symmetric_differences.statement}
For convex sets $C, D \subseteq X$, the points in $F_C(D)$ are precisely those that are part of a symmetric difference:
\begin{equation*}
F_C(D) = \{c \in C : c-d = -\lambda(c'-d') \text{ for some } \lambda > 0,\ c' \in C,\ d, d' \in D\}.
\end{equation*}
\end{proposition}
\begin{proof}[\hypertarget{loc:generated_faces_are_formed_by_symmetric_differences.proof}{}Proof of \Cref{loc:generated_faces_are_formed_by_symmetric_differences.statement}]

Write $T_a := T_a(C,D) = \gamma_1^{-1}(\lin(K-G))$, so that $F_C(D) = T_a \cap C$ by \Cref{loc:characterization_of_the_generated_face_of_a_convex_set.remark}. By \Cref{loc:pairs_forming_symmetric_differences.statement}, the r.h.s. is also characterized as $T_a \cap C$, up
to guaranteeing the existence of a $d  \in D \cap T_a$ to pair with every $c  \in C \cap T_a$ in \Cref{loc:pairs_forming_symmetric_differences.statement}.
This follows if we show that $T_a \neq \varnothing \implies C \cap D \neq \varnothing$, since we can then pick any $d \in C \cap D \subseteq D \cap T_a$.

We show the contrapositive, that $C \cap D = \varnothing \implies \lin(K-G)= \{0\}$. This follows from equivalence with point 3 of \Cref{loc:cross_lineality.statement}, which with \Cref{loc:the_generated_supporting_subspace_is_a_cross:lineality.statement} says that 
\begin{align*}
\lin(K-G)
&= \lin(K - (K \cap G)) + \lin(G- (K \cap G))\\
&= \lin(K- \{0\}) + \lin(G-\{0\})\\
&= \{0\},
\end{align*}
where the third equality follows from $K$ and $G$ being homogenization cones, and therefore pointed.
\end{proof}
\begin{proof}[\hypertarget{loc:affine_cross:lineality.proof}{}Proof of \Cref{loc:affine_cross:lineality.statement}]

If $C \cap D =  \emptyset$, then all points are the same set because they are
all the empty set. Next we show equality between all points in the case that $C \cap D \neq \varnothing$.

$1 = 3$:
By \Cref{loc:characterization_of_the_separating_flats.statement} point 1,
the separating flats of $(C,D)$ are exactly the preimages $\gamma_1^{-1}(U)$ of the separating subspaces $U$ of $(K,G)$. By \Cref{loc:cross_lineality.statement} point 2, $\lin(K-G)$ is the minimum separating subspace. Since $\gamma_1^{-1}$ preserves inclusion, it sends this minimum to the minimum separating flat, which is point 1. We note this argument holds even when $C \cap D = \varnothing$.

$3 = 2$:
By \Cref{loc:characterization_of_the_separating_flats.statement} point 4, the map $\tilde{L} \mapsto \tilde{L} + C \cap D$ sends the supporting subspaces of $\cone(C-D)$ onto the separating flats of $(C,D)$, and it preserves inclusion. By \Cref{loc:characterization_of_the_lineality.statement} point 4, $\lin\cone(C-D)$ is the minimum supporting subspace, so its image $\lin\cone(C-D) + C \cap D$ is the minimum separating flat.

$1 = 5$:
We compute from point 1:
\begin{align*}
\gamma_1^{-1}(\lin(K-G))
&= \gamma_1^{-1}\Span\bigl(F_K(G) \cup F_G(K)\bigr)\\
&= \gamma_1^{-1}\Span\bigl(\convcone(F_K(G) \cup F_G(K))\bigr)\\
&= \aff \gamma_1^{-1}\bigl(\convcone(F_K(G) \cup F_G(K))\bigr)\\
&= \aff\bigl(\gamma_1^{-1}(F_K(G)) \cup \gamma_1^{-1}(F_G(K))\bigr)\\
&= \aff\bigl(F_C(D) \cup F_D(C)\bigr),
\end{align*}
which is point 5. The first equality is \Cref{loc:cross_lineality.statement} point 3 with the identification of $H_K(G) = \Span(F_K(G))$ by \Cref{loc:generated_face_spans_the_lineality.statement}, with the fact that 
$\Span(F_K(G)) + \Span(F_G(K)) =  \Span(F_K(G) \cup F_G(K))$. The third is \Cref{loc:span_dehomogenizes_to_the_affine_hull.statement}, applied to the homogenization cone $\convcone(F_K(G) \cup F_G(K)) \subseteq (X \times \mathbb{R}_{++}) \cup \{0\}$.
The fifth is \Cref{loc:characterization_of_the_generated_face_of_a_convex_set.statement} point 2.

$5 = 4$:
First consider that
\begin{align*}
\aff(F_C(D)) &= \aff\bigl(F_C(D) \cup (C \cap D)\bigr) \\
&= \aff\bigl(F_C(C \cap D) \cup F_{C \cap D}(C)\bigr) \\
&= \lin\cone(C - (C \cap D)) + C \cap D =: L_C + C \cap D,
\end{align*}
where we defined $L_C := \lin \cone(C-(C \cap D))$ in the last line. The first equality holds because $C \cap D \subseteq F_C(D)$ (follows from \Cref{loc:characterization_of_the_generated_face_of_a_convex_set.statement} point 4). The second equality uses the fact that $F_C(C \cap D) = F_C(D)$ by equivalence of points 3 and 4 of \Cref{loc:characterization_of_the_generated_face_of_a_convex_set.statement}, and also that $F_{C \cap D}(C) = C \cap D$ because $C \cap D \subseteq C$. The third equality uses the equivalence between
points 5 and 2 shown above, applied to the convex sets $C$ and $C \cap D$.

Similarly, we have $\aff(F_D(C)) = L_D + C \cap D$. The subspaces $L_C$ and $L_D$ both contain $C \cap D - C \cap D$, therefore the term $(+ C \cap D)$ is a common affine offset applied to both. Therefore their affine hull is the flat of direction $L_C + L_D$ through $C \cap D$, i.e.,
\begin{equation*}
\aff(F_C(D) \cup F_D(C)) = L_C + L_D + C \cap D,
\end{equation*}
which is point 4.

$6 = 5$:
By \Cref{loc:generated_faces_are_formed_by_symmetric_differences.statement},
the set of points forming symmetric differences is
precisely $F_C(D) \cup F_D(C)$.
\end{proof}
\subsubsection{The generated supporting flat}
\label{loc:body.lattice_minima.the_joint_supporting_subspace.the_generated_supporting_flat}
When considering nested convex sets $S \subseteq C \subseteq X$, the joint supporting subspace yields the generated supporting flat.
\begin{proposition}[The joint supporting subspace is the generated supporting flat]
\label{loc:the_joint_supporting_subspace_is_the_generated_supporting_flat.statement}
Let $C \subseteq X$ be a convex set and let $S \subseteq C$ be a convex set. Then $T_a(C,S)$ is the minimum supporting flat of $C$ containing $S$.
\end{proposition}
The \hyperlink{loc:the_joint_supporting_subspace_is_the_generated_supporting_flat.proof}{proof} in the appendix dehomogenizes \Cref{loc:separation_reduces_to_containment.statement}.
\begin{proposition}[Generated face spans the generated supporting flat]
\label{loc:generated_face_spans_the_generated_supporting_flat.statement}
For convex sets $C, D \subseteq X$,
\begin{equation*}
\aff(F_C(D)) = T_a(C, C \cap D).
\end{equation*}
\end{proposition}
\begin{proof}[\hypertarget{loc:generated_face_spans_the_generated_supporting_flat.proof}{}Proof of \Cref{loc:generated_face_spans_the_generated_supporting_flat.statement}]

\begin{align*}
T_a(C, C \cap D) &= \aff\bigl(F_C(C \cap D) \cup F_{C \cap D}(C)\bigr)\\
&= \aff\bigl(F_C(D) \cup (C \cap D)\bigr)\\
&= \aff(F_C(D)).
\end{align*}
The first equality is \Cref{loc:affine_cross:lineality.statement} point 5 at the pair $(C, C \cap D)$. The second is points 3 and 4 of \Cref{loc:characterization_of_the_generated_face_of_a_convex_set.statement} for $F_C(C \cap D) = F_C(D)$, and its point 1 for $F_{C \cap D}(C) = C \cap D$, a set being its own generated face. The third is $C \cap D \subseteq F_C(D)$, again by point 1.
\end{proof}
\section{Applications}
\label{loc:body.applications}
\subsection{Faces of differences}
\label{loc:body.applications.faces_of_differences}
In \Cref{loc:decomposition_of_faces_of_difference_of_convex_cones.statement} we showed that the faces of a difference of two convex cones decompose into a difference of faces of the original convex cones in a manner structured by separating subspaces. Here we derive the equivalent for convex sets, for which there is not as much structure.
We must resort to minimization, instead of separating subspaces; separating
flats are not sufficiently powerful to obtain the full result.
\begin{theorem}[Decomposition of faces of the difference of convex sets]
\label{loc:decomposition_of_faces_of_the_difference_of_convex_sets.statement}
Let $C, D \subseteq X$ be convex sets, let $F \triangleleft (C-D)$ be a non-empty face, and let
$P \subseteq X$ be a preorder for which $F = \min_{P} (C-D)$.
Then $F_C := \min_{P}(C)$ and $F_D := \max_{P}(D)$ are faces of $C$ and $D$ respectively, and $F = F_C-F_D$.
Further, such a preorder $P$ always exists: for instance, the choice $P := \cone((C-D)-F)$ has the required property. 
\end{theorem}
This result can be alternatively stated for a Minkowski sum,
in which case we instead have $F = \min_{P}(C) + \min_{P}(D)$.
\Cref{loc:decomposition_of_faces_of_the_difference_of_convex_sets.statement} generalizes \cite[Theorem 1.7.5]{schneider_ConvexBodiesBrunn_2013}
across three axes: to infinite dimensions, to non-compact convex sets,
and to faces that are not necessarily exposed.
\begin{proof}[\hypertarget{loc:decomposition_of_faces_of_the_difference_of_convex_sets.proof}{}Proof of \Cref{loc:decomposition_of_faces_of_the_difference_of_convex_sets.statement}]

That the preorder $P := \cone((C-D)-F)$ is such that $F = \min_{P} (C-D)$
follows from \Cref{loc:characterization_of_the_generated_face_of_a_convex_set.statement} points 1 and 3, as $F$ is the face of $C-D$ generated by itself.

In
the following, every inequality, $\min$ and $\max$ are taken in the preorder $P$. By 
\Cref{loc:characterization_of_the_faces_of_convex_sets.statement} point 3,
both $\min C$ and $\max D$ are faces, so it only remains to show that
\begin{equation*}
\min C - \max D = \min (C-D).
\end{equation*}

$\subseteq$: Let $c \in C, d \in D, c' \in \min C$, and
$d' \in \max D$. We have that $c'-d' \leq c'-d \leq c-d$,
which shows that $\min C - \max D \leq C - D$, and therefore that
$\min C - \max D \subseteq \min (C-D)$.

$\supseteq$: For any $c \in C, d \in D$ such that $c-d \notin \min C-\max D$,
consider the case where $c \notin \min C$. For any $c' \in  \min C$,
we have $c'- d < c-d$, so $c-d \notin \min(C-D)$. We have the same result when considering instead 
$d \notin \max D$. This shows that
$(\min C - \max D)^c \subseteq (\min (C-D))^c$, so we have shown the containment and hence the result.
\end{proof}
\subsection{Faces of intersections}
\label{loc:body.applications.faces_of_intersections}
The next result resolves the open problem 5.10 stated in
\cite{weisNoteFacesConvex2025}.
\begin{theorem}[Face of intersections is the intersection of faces]
\label{loc:face_of_intersections_is_the_intersection_of_faces.statement}
For two convex sets $C, D \subseteq X$, for any face $A \triangleleft (C \cap D)$
we have 
\begin{equation*}
A = F_C(A) \cap F_D(A).
\end{equation*}
\end{theorem}
In \cite{weisNoteFacesConvex2025}, the same decomposition is shown under the
condition that the face $A$ itself have a non-empty intrinsic core (also known
as the relative algebraic interior), which
holds in particular whenever $A$ is finite-dimensional. The general case was
left open.

Our proof method is to express the generated faces as 
$F_C(A) = \min_{\cone(C-A)}C$, by
\Cref{loc:characterization_of_the_generated_face_of_a_convex_set.statement} point 3, 
a characterization which associates the generated face with the preorder $\cone(C-A)$. Then we show that the intersection of generated faces is obtained by the intersection of their associated preorders. 
\begin{lemma}[Minimizers for intersection of preorders]
\label{loc:minimizers_for_intersection_of_preorders.statement}
For preorders $K, G \subseteq X$ and a convex set $S \subseteq X$,
\begin{equation*}
\min_{K}S \cap \min_{G}S = \min_{K \cap G} S.
\end{equation*}
\end{lemma}
\begin{proof}[\hypertarget{loc:minimizers_for_intersection_of_preorders.proof}{}Proof of \Cref{loc:minimizers_for_intersection_of_preorders.statement}]

For $x \in S$, minimality is a comparison with every point of $S$,
\begin{equation*}
x \in \min_{K \cap G}S \iff x \le_{K \cap G} S,
\end{equation*}
and
\begin{equation*}
x \le_{K \cap G} y \iff x \le_{K} y \land x \le_{G} y,
\end{equation*}
therefore
\begin{equation*}
x \in \min{}_{K \cap G}S \iff x \le_{K} S \land x \le_{G} S \iff x \in \min{}_{K}S \cap \min{}_{G}S.
\end{equation*}
\end{proof}
\begin{lemma}[Difference cone with intersection of containing convex sets]
\label{loc:difference_cone_with_intersection_of_containing_convex_sets.statement}
Let $C, D, F \subseteq X$ be convex sets such that $F \subseteq C \cap D$.
Then
\begin{equation*}
\cone((C \cap D) - F) = \cone(C-F) \cap \cone(D-F)
\end{equation*}
\end{lemma}
\begin{proof}[\hypertarget{loc:difference_cone_with_intersection_of_containing_convex_sets.proof}{}Proof of \Cref{loc:difference_cone_with_intersection_of_containing_convex_sets.statement}]

$(\subseteq)$: Trivial.

$(\supseteq)$: Let $x \in \cone(C-F) \cap \cone(D-F)$, and pick $\lambda_1, \lambda_2 > 0$
with $\lambda_1 x \in C-F$ and $\lambda_2 x \in D-F$. Both $C-F$ and $D-F$ are convex
and contain $0$, hence closed under shrinking, so $w := \min(\lambda_1, \lambda_2)x$
lies in both: there exist elements $c \in C$, $d \in D$ and $f_1, f_2 \in F$ such that
\begin{equation*}
w = c - f_1 = d - f_2.
\end{equation*}
We claim that $\frac{w}{2} \in (C \cap D) - F$, which suffices to show the containment since its multiple $x$ is then in $\cone((C \cap D) - F)$.
The claim is verified by taking the midpoint $f := \frac{f_1 + f_2}{2} \in F$, with which $\frac{w}{2}= (w /2 + f) - f  \in C \cap D - F$, because
\begin{equation*}
f + \frac{w}{2} = \frac{c + f_2}{2} \in C, \qquad f + \frac{w}{2} = \frac{f_1 + d}{2} \in D.
\end{equation*}
\end{proof}
\begin{proof}[\hypertarget{loc:face_of_intersections_is_the_intersection_of_faces.proof}{}Proof of \Cref{loc:face_of_intersections_is_the_intersection_of_faces.statement}]

For $A = \varnothing$ the result holds because the three faces are empty. When $A \neq \varnothing$, the result holds because 
\begin{align*}
F_C(A) \cap F_D(A)  = &\min{}_{\cone(C-A)}C \cap \min{}_{\cone(D-A)}D\\
= &\min_{\cone(C-A)} (C \cap D) \cap  \min_{\cone(D-A)}(C \cap D)\\
 = &\min_{\cone(C-A) \cap \cone(D - A)}(C \cap D)\\
=& \min_{\cone((C \cap D) - A)}(C \cap D) =  F_{C \cap D}(A) = A,
\end{align*}
where the first equality corresponds to the characterization of generated faces of \Cref{loc:characterization_of_the_generated_face_of_a_convex_set.statement} point 3, the second equality holds because every point in the intersection of the minimum sets is in $C \cap D$, the third holds by \Cref{loc:minimizers_for_intersection_of_preorders.statement}, the fourth by \Cref{loc:difference_cone_with_intersection_of_containing_convex_sets.statement}, and the last because $A \triangleleft C \cap D$.
\end{proof}
\subsection{Facial reduction and lexicographic maps}
\label{loc:body.applications.facial_reduction_and_lexicographic_maps}
The lexicographic characterization of faces of convex sets provided
in \cite[Theorem 2]{martinez-legaz_LEXICOGRAPHICALCHARACTERIZATIONFACES_}
says that for any convex set $C \subseteq \mathbb{R}^n$, any face
$F \triangleleft C$ can be expressed as 
\begin{equation*}
F = \argmax_{x  \in  C} \, Ax,
\end{equation*}
for some matrix $A \in \mathbb{R}^{m \times n}$ with $m \le n$, 
where the argmax is over the lexicographic order in $\mathbb{R}^m$. 
This characterization is extended to infinite dimensions with step-affine
functions in \cite{gorokhovikCharacterizationsFacesConvex2026}, though
we presently only discuss the finite-dimensional case.

The lexicographic characterization was originally obtained from the concept of lexicographic separation of convex sets. In
an immediately preceding work, \textcite{martinez-legazLexicographicalSeparationRn1987} show that any
two disjoint convex sets can be strictly separated in the lexicographic order
(the set images are strictly ordered in the lexicographic order; $A(C) < A(D)$).
This ``true" separation for any pair of disjoint convex sets is unique to lexicographic orders.
Using this framework,
it is shown in \cite{martinez-legaz_LEXICOGRAPHICALCHARACTERIZATIONFACES_}
that any face can be obtained by lexicographically separating 
a face $F \triangleleft C$ from its complement
$C \setminus F$, which is also convex. This positions lexicographic maps as a seemingly unique
way of characterizing the faces of convex sets.

In the language of the present work,
the characterization uses a linear oriented map $A:X \to (Y, Y_+)$ where
$Y_+$ is the positive cone of the lexicographic order in $Y$. We call such
maps and their affine counterparts lexicographic maps.
We thus situate lexicographic maps as members of the larger class of oriented maps.
Our results make it manifest that lexicographic orders have no privileged access to faces, as
we showed that any preorder obtains a face via minimization.
On the other hand, the fact that lexicographic preorders are in themselves
sufficient to recover any face of any convex set is a nontrivial fact that
does not follow from our theory of preorders thus far. Nonetheless, our machinery provides convenient
language to show this interesting fact. In this section we provide
a simple, alternative proof for this fact, by connecting
lexicographic maps with the notion of facial reduction in conic programming,
recently extended to bilateral facial reduction between any two convex 
sets~\cite{scottBilateralFacialReduction2026, linFacialReductionNice2025}.
We define the lexicographic product for oriented maps, from which it is possible to build lexicographic maps from oriented functionals. 
This product is analogous to the lexicographic product of cones in the theory of ordered groups~\cite{fuchsPartiallyOrderedAlgebraic2014}, and generalizes the procedure
of adding rows to the matrix $A$ in~\cite{martinez-legaz_LEXICOGRAPHICALCHARACTERIZATIONFACES_}.
\begin{definition}[Lexicographic product of oriented maps]
\label{loc:lexicographic_product_of_oriented_maps.statement}
Let $\phi_1:X \to (Y_1, Y_{1+})$ and $\phi_2:X \to (Y_2, Y_{2+})$ be either both linear or both affine oriented maps. The lexicographic product forms the oriented map $\phi_1 \rtimes \phi_2:X \to (Y_1 \times Y_2, Y_+)$, where the product takes values $\forall x \in  X,  \phi_1 \rtimes \phi_2(x) = \phi_1 \times \phi_2 (x) = (\phi_1(x),  \phi_2(x))$, and the codomain of the product takes the partial order $Y_+ := ((Y_{1+} \setminus \{0\}) \times Y_2) \cup (\{0\} \times Y_{2+})$. 
\end{definition}
The kernel of the product is the intersection of the kernels: $\ker(\phi_1 \rtimes \phi_2) = \ker\phi_1 \cap \ker\phi_2$.
The lexicographic product is associative: that
$\phi_1 \rtimes \phi_2 \rtimes \phi_3$ is positive on a point $x$ corresponds to the statement that 
\begin{equation*}
\phi_1(x) \in Y_{1+} \setminus \{0\} \,  \lor \, \bigl( \phi_1(x) = 0 \ \land\ \phi_2(x) \in Y_{2+} \setminus \{0\} \bigr) \, \lor \, \bigl( \phi_1(x) = 0 \ \land\ \phi_2(x) = 0 \ \land\ \phi_3(x) \in Y_{3+} \bigr),
\end{equation*}
a condition which can be seen to correspond to both $(\phi_1 \rtimes \phi_2) \rtimes \phi_3$ and $\phi_1 \rtimes (\phi_2 \rtimes \phi_3)$.
The lexicographic maps of~\cite{martinez-legaz_LEXICOGRAPHICALCHARACTERIZATIONFACES_}
are then products $\phi := \psi_1 \rtimes \ldots \rtimes \psi_m$ of oriented functionals on $W$, where each $\psi_i$ is non-trivial on the kernel $\ker(\psi_1 \rtimes \ldots \rtimes \psi_{i-1}) = \bigcap_{j<i}\ker\psi_j$ of the preceding partial product.
We now formalize bilateral facial reduction~\cite{scottBilateralFacialReduction2026}
as the procedure of iteratively building lexicographic products while 
maintaining separation of two given convex sets $C, D$ at every step.
This perspective unifies the lexicographic characterization of faces~\cite{martinez-legaz_LEXICOGRAPHICALCHARACTERIZATIONFACES_}
and facial reduction~\cite{drusvyatskiyManyFacesDegeneracy2017} under our framework.
We first provide a short proof for the central result of \cite{scottBilateralFacialReduction2026},
that the joint supporting subspace is the constraint obtained by the iterative reduction,
and then show that the lexicographic characterization of faces follows as a corollary. We begin with a lemma in
the conic setting.
\begin{lemma}[Alternatives for a separating linear oriented map]
\label{loc:alternatives_for_a_separating_linear_oriented_map.statement}
Let $K,G \subseteq W$ be convex cones for $W$ a finite-dimensional
vector space. For any oriented linear map $\phi$ such that $\phi(K) \geq \phi(G)$, at least one of the following holds:
\begin{enumerate}
\item $K \cap \ker \phi = \{0\}$ or $G \cap \ker \phi = \{0\}$,
\item $\ker \phi = \lin(K-G)$,
\item $\exists \psi:W \to (\mathbb{R}, \mathbb{R}_+)$ an oriented linear functional, non-zero on $\ker\phi$, such that $(\phi \rtimes \psi)(K) \geq (\phi \rtimes \psi)(G)$.
\end{enumerate}
\end{lemma}
In the above lemma, finite dimensionality is only needed to guarantee separation by functionals. It can be relaxed in favor of interiority conditions that allow for the use of the Hahn-Banach theorem.
\begin{proof}[\hypertarget{loc:alternatives_for_a_separating_linear_oriented_map.proof}{}Proof of \Cref{loc:alternatives_for_a_separating_linear_oriented_map.statement}]

Assume that cases 1 and 2 fail. Narrowing the ambient space to $Z := \ker\phi$, we must show the existence of a non-trivial separating oriented functional $\psi:\ker\phi \to (\mathbb{R}, \mathbb{R}_+)$, i.e., $\psi(G') \leq \psi(K')$ for $K' := K \cap Z$ and $G' := G \cap Z$. The existence of $\psi$ is guaranteed by the Hahn-Banach theorem in finite dimensions, so long as $\lin(K'-G') \subsetneq \ker\phi$.

The condition holds: $\lin(K'-G') \subseteq \lin(K-G) \subsetneq \ker\phi$, by monotonicity of the lineality and the failure of case 2. It guarantees separation, as follows. Take $z \in \ker\phi \setminus \lin(K'-G')$. Then one of $z, -z$ lies outside the convex cone $P := K' - G'$, say $-z \notin P$. In finite dimensions, the disjoint convex sets $P$ and $\{-z\}$ are separated by a non-zero functional $\psi$: $\psi(P) \geq \psi(-z)$. Since $P$ is a cone, this lower bound forces $\psi(P) \geq 0$. Then $\psi(P) = \psi(K' - G') \geq 0  \implies \psi(K') \geq \psi(G')$.
Extend $\psi$ linearly to $W$, which leaves it non-zero on $\ker\phi$. The separation $\psi(G') \leq \psi(K')$ is sufficient to show that $(\phi \rtimes \psi)(K) \geq (\phi \rtimes \psi)(G)$.
\end{proof}
Dehomogenization yields the corresponding statement for affine oriented maps, where the role of the lineality is played by the joint supporting subspace of the two convex sets.
\begin{lemma}[Alternatives for a separating affine oriented map]
\label{loc:alternatives_for_a_separating_affine_oriented_map.statement}
 Let $C, D \subseteq X$ be convex sets for $X$ a finite-dimensional
vector space. For any affine oriented map $f:X \to (Y,Y_+)$ such that
$f(D) \le 0 \le f(C)$, at least one of the following holds:
\begin{enumerate}
\item $f(D) < f(C)$,  making $f$ a certificate showing that $C \cap D = \varnothing$,
\item $\ker f = T_a(C,D)$,
\item $\exists \psi:X \to (\mathbb{R}, \mathbb{R}_+)$ an affine oriented functional, non-constant on $\ker f$, such that $(f \rtimes \psi)(D) \le 0 \le (f \rtimes \psi)(C)$.
\end{enumerate}
\end{lemma}
The proof is a dehomogenization argument,
with the main difficulty in showing that the existence of a separating non-trivial functional
in the homogenization space translates to a non-trivial separating affine functional. We leave \hyperlink{loc:alternatives_for_a_separating_affine_oriented_map.proof}{this proof} to the appendix.
The above lemma specifies the behaviour of bilateral facial reduction at each step. The next result describes the full procedure, using the fact that
termination is guaranteed by the dimensions being finite.
\begin{theorem}[Bilateral facial reduction with lexicographic products]
\label{loc:bilateral_facial_reduction_with_lexicographic_products.statement}
Let $C, D \subseteq X$ be convex sets for $X$ a finite-dimensional vector space. 
Let $\phi_0: X \to (\{ 0 \},\{ 0 \})$. We run the following loop, increasing $m$
at each step.

For  $m \in \{1, 2, \ldots\}$:
\begin{itemize}
\item If $\phi_{m-1}$ satisfies points 1 or 2 of \Cref{loc:alternatives_for_a_separating_affine_oriented_map.statement}, return $\phi_{m-1}$.
\item Otherwise, let $\phi_m :=  \phi_{m-1} \rtimes \psi_m$, for the affine functional $\psi_m$ on $X$, non-constant on $\ker\phi_{m-1}$, guaranteed to exist by \Cref{loc:alternatives_for_a_separating_affine_oriented_map.statement} point 3.
\end{itemize}

This process terminates for some $m \le \dim(X)$, yielding an oriented
affine map separating $C$ from $D$ that satisfies \Cref{loc:alternatives_for_a_separating_affine_oriented_map.statement} points 1 or 2.
\end{theorem}
\begin{proof}[\hypertarget{loc:bilateral_facial_reduction_with_lexicographic_products.proof}{}Proof of \Cref{loc:bilateral_facial_reduction_with_lexicographic_products.statement}]

That $\phi_{m-1}$ separates $C$ and $D$ follows by induction: 
$\phi_0$ is trivially separating, and any generated map preserves this
property by \Cref{loc:alternatives_for_a_separating_affine_oriented_map.statement} point 3. 

That termination must occur follows from the fact that the dimension of the kernel
decreases by $1$ at each step: $\ker\phi_m = \ker\phi_{m-1} \cap \ker\psi$, and 
the functional in point 3 of \Cref{loc:alternatives_for_a_separating_affine_oriented_map.statement} is guaranteed to be non-constant on $\ker\phi_{m-1}$, so its zero set there is a hyperplane of that flat.
\end{proof}
We next show that the characterization of faces by lexicographic maps is a corollary of
bilateral facial reduction.
\begin{corollary}[Lexicographic characterization of faces]
\label{loc:lexicographic_characterization_of_faces.statement}
Let $X$ be a finite-dimensional vector space and $C \subseteq X$ a convex set. For any non-empty face $F \triangleleft C$ there is an affine lexicographic map $\phi:X \to (Y, Y_+)$ such that $\phi(C) \geq 0$ and $F = C \cap \ker\phi$.
\end{corollary}
\begin{proof}[\hypertarget{loc:lexicographic_characterization_of_faces.proof}{}Proof of \Cref{loc:lexicographic_characterization_of_faces.statement}]

Apply \Cref{loc:bilateral_facial_reduction_with_lexicographic_products.statement} with $D = F$. The resulting map cannot satisfy point 1 of \Cref{loc:alternatives_for_a_separating_affine_oriented_map.statement}, since $F \subseteq C$ is nonempty. Then it must satisfy point 2, so
we obtain an affine lexicographic map $\phi$ with $\phi(F) \le 0 \le \phi(C)$
such that $\ker \phi = \lin\cone(C-F) + F = T_a(C,F)$. Therefore, $\ker \phi \cap C = F_C(F) = F$
by \Cref{loc:characterization_of_the_generated_face_of_a_convex_set.remark}, and because $F$ is a face of $C$.
\end{proof}
\section{Conclusion}
\label{loc:body.conclusion}
We characterized many types of objects in this work, yet they are all derived from oriented linear maps $\psi:W \to (Y, Y_+)$.
Dehomogenization, formalized as precomposition with $\gamma_1$, yields the affine oriented maps. Orthogonally, we apply restrictions to the oriented maps: positivity on one convex cone, and in the case of separation, an additional requirement of negativity on a second convex cone. From these maps, and subject to these restrictions, we obtain preorders, supporting subspaces and faces.
These other objects are all obtained ``downstream" of oriented maps: for an
oriented map $\psi$ positive on a convex cone $K$,
\begin{equation*}
\psi \mapsto (P := \psi^{-1}(Y_+)) \mapsto (U := \lin(P)) \mapsto (F := U \cap K).
\end{equation*}
These are the operations detailed in \Cref{loc:body.positivity_minimization_and_faces}.

\Cref{loc:body.lattice_minima} can be understood as reversing this process:
\begin{equation*}
F \mapsto (U := \Span(F)) \mapsto (P := K - U) \mapsto \pi_P,
\end{equation*}
where in the last step the quotient map $\pi_P$ is obtained from $P$ as
in \Cref{loc:preorders_are_characterized_by_oriented_maps.statement}.
These operations go ``upstream" by building the ``smallest" upstream element
that yields the downstream element. Generated elements are obtained
via ``round-trips". For example, to obtain the generated supporting subspace
from a given subspace, we build the minimum upstream object and then go back
downstream. For a subspace $L \subseteq W$,
\begin{equation*}
L \mapsto (P := K - L) \mapsto \lin(P),
\end{equation*}
which gives the formula $H_K(L) := \lin(K-L)$ of \Cref{loc:separation_reduces_to_containment.statement}.

\paragraph{Summary and outlook}

\label{loc:body.conclusion.summary_and_outlook}
Motivated by the understanding of convex sets as slices of convex cones in the homogenization space, we introduced affine oriented maps on convex sets, from which the notions of faces, supporting subspaces, and generalized separation emerge. By doing this, we unify exposed faces, hyperplane separation, lexicographically exposed faces, faces generated at a point, the joint supporting subspace, and bilateral facial reduction all under one framework.
Downstream of our framework, three geometric consequences follow easily. We show that any face of an intersection of two convex sets is an intersection of faces of the individual convex sets (\Cref{loc:face_of_intersections_is_the_intersection_of_faces.statement}), which resolves the open problem of \cite{weisNoteFacesConvex2025}. We also show that any face of a difference of convex sets is the difference of faces (\Cref{loc:decomposition_of_faces_of_the_difference_of_convex_sets.statement}), and this without compactness, exposedness, or finite dimensions. Finally, our framework provides a natural formulation and simple proofs for bilateral facial reduction (\Cref{loc:bilateral_facial_reduction_with_lexicographic_products.statement}), from which
the lexicographic characterization of faces follows as a corollary (\Cref{loc:lexicographic_characterization_of_faces.statement}).

We believe that the use of oriented maps provides a novel, powerful, motivated, yet simple unifying method in the analysis of convex sets. We believe it has the potential to form a valuable foundation for a qualification-free theory of convex analysis, including results such as those introduced in~\cite{scottBilateralFacialReduction2026}.
\section{Acknowledgements}
\label{loc:body.acknowledgements}
The author was financially supported by the University of British Columbia during the making of this manuscript. 
AI tools (Claude Fable 5, Opus 5 and Grok 4.6) were used when making this work, helping with ideation, proof writing, searching the literature, and checking for errors. All writing was edited manually before submission, and almost all insights,
proofs, and paper structure originated from the author.
The author has no competing interests to declare, and no data or code were created for this work.
\printbibliography
\appendix
\section{Appendix}
\begin{proof}[\hypertarget{loc:characterization_of_the_lineality.proof}{}Proof of \Cref{loc:characterization_of_the_lineality.statement}]

$1 = 2$: Immediate.

$3 = 1$: For any subspace $U \subseteq K$, any point in $U$ has an additive inverse: $\{ u, -u \}  \in U \subseteq K$ therefore $u \in \lin(K)$, which shows that $U \subseteq \lin(K)$. That $\lin(K)$ is a subspace follows from similarity with $0$ being preserved by vector space operations.

$2 = 4$: With $\lin(K)$ defined with point $2$, $\lin(K)$ is a supporting subspace of $K$ by \Cref{loc:characterization_of_supporting_subspaces.statement} point 4 considered with the preorder $K$. Containment of $\lin(K)$ (as defined in point 2) in an arbitrary supporting subspace $U$ of $K$ follows from the characterization of supporting subspaces \Cref{loc:characterization_of_supporting_subspaces.statement} point 2, by taking $z = 0$, since $0$ is contained in any supporting subspace.

$2 = 5$: Since $\lin(K)$ is a supporting subspace of $K$ by point 4, and $\lin(K) \subseteq K$ by point 2, we have $\lin(K) = K \cap \lin(K) \triangleleft K$ by \Cref{loc:characterization_of_supporting_subspaces.statement} point 5. Containment of $\lin(K)$ in an arbitrary face $F$ of $K$ follows from \Cref{loc:characterization_of_the_faces_of_convex_cones.statement} point 2, again by taking $z = 0$, since $0$ is contained in any face.
\end{proof}
\subsection{Dehomogenization proofs}
\label{loc:appendix.dehomogenization_proofs}
\begin{proof}[\hypertarget{loc:dehomogenization_is_a_bijection_between_sets_and_homogenization_cones.proof}{}Proof of \Cref{loc:dehomogenization_is_a_bijection_between_sets_and_homogenization_cones.statement}]

We show that $\gamma_1^{-1}$ and $\cone\circ\gamma_1$ are mutually inverse, by checking each composition.

Let $K$ be a homogenization cone. Every $k \in K$ is of the form $\lambda\gamma_1(c)$ for some $c \in \gamma_1^{-1}(K)$:
the origin with $\lambda = 0$, and otherwise $k = (x,\lambda)$ with $\lambda > 0$ and $c = x/\lambda$,
since $\gamma_1(c) = k/\lambda \in K$. Thus $K \subseteq \cone\circ\gamma_1(\gamma_1^{-1}(K))$. The reverse inclusion
holds because $\gamma_1(\gamma_1^{-1}(K)) \subseteq K$ and $\cone(K) = K$, so
\begin{equation*}
K = \cone\circ\gamma_1(\gamma_1^{-1}(K)).
\end{equation*}

Conversely, let $C \subseteq X$ and $K := \cone\circ\gamma_1(C)$. The inclusion $C \subseteq \gamma_1^{-1}(K)$ is $\gamma_1(C) \subseteq K$. For the reverse, any point of $K$ off $\gamma_1(C)$ has height other than one, so $\gamma_1^{-1}$ sees only $\gamma_1(C)$:
\begin{equation*}
\gamma_1^{-1}(\cone\circ\gamma_1(C)) = C.
\end{equation*}
\end{proof}
\begin{proof}[\hypertarget{loc:span_dehomogenizes_to_the_affine_hull.proof}{}Proof of \Cref{loc:span_dehomogenizes_to_the_affine_hull.statement}]

The subspaces are the kernels of the linear maps and the non-empty flats the kernels of the affine maps, with $\ker \psi \supseteq H \iff \psi(H) = 0$. Each hull, being the smallest such set containing its argument, is therefore an intersection of kernels,
\begin{equation*}
\Span H = \bigcap_{\psi(H) = 0} \ker \psi, \qquad \aff S = \bigcap_{f(S) = 0} \ker f.
\end{equation*}

Being a preimage, $\gamma_1^{-1}$ commutes with intersections, and each term dehomogenizes on its own, $\gamma_1^{-1}(\ker \psi) = \ker f$ by \Cref{loc:dehomogenization_of_kernels.statement}. Hence
\begin{align*}
\gamma_1^{-1}(\Span H) = \bigcap_{\psi(H) = 0} \gamma_1^{-1}(\ker \psi)
= \bigcap_{\psi(H) = 0} \ker f
= \bigcap_{f(S) = 0} \ker f = \aff S,
\end{align*}
where the third equality reindexes along the bijection $\psi \mapsto f$ onto the affine maps. It is valid because $\psi(H) = 0 \iff f(S) = 0$ by \Cref{loc:preservation_of_orientation_in_dehomogenization.statement} applied to $\psi$ and to $-\psi$, which gives $\psi(H) \ge 0 \iff f(S) \ge 0$ and $\psi(H) \le 0 \iff f(S) \le 0$.
\end{proof}
\begin{proof}[\hypertarget{loc:bijection_of_faces_containing_a_subset.proof}{}Proof of \Cref{loc:bijection_of_faces_containing_a_subset.statement}]

The cone $T$ and faces of $K$ are all homogenization cones (i.e., contained in $(X  \times \mathbb{R}_{++}) \cup \{ 0 \}$) because they are subcones of $K$. Therefore, they are in the domain of $\gamma_1^{-1}$ as a lattice isomorphism from convex cones contained in $(X \times \mathbb{R}_{++})\cup\{0\}$ to convex sets as guaranteed by \Cref{loc:dehomogenization_is_a_lattice_isomorphism_between_convex_sets_and_homogenization_cones.statement}. From this, we have that for any face $F$ of $K$,
\begin{equation*}
T \subseteq F \iff S \subseteq \gamma_1^{-1}(F),
\end{equation*}
allowing for a restriction of $\gamma_1^{-1}$ to an isomorphism between only those sets containing $T$ and $S$ respectively.
\Cref{loc:dehomogenization_is_a_bijection_on_the_faces.statement} states that $\gamma_1^{-1}$ 
also restricts to faces of $K$ and $C$; applying both restrictions yields the result.
\end{proof}
\begin{proof}[\hypertarget{loc:differences_in_the_lifting.proof}{}Proof of \Cref{loc:differences_in_the_lifting.statement}]

$(\supseteq)$: Note that $C \times \{1\} \subseteq K$
and $D \times \{1\} \subseteq G$, so we have that 
\begin{equation*}
(C-D) \times \{0\} = C \times \{1\} - D \times \{1\} \subseteq K-G.
\end{equation*}
And since $K-G$ is a cone,
$\cone(C-D) \times \{0\} \subseteq K-G$.

$(\subseteq)$: Any elements $k \in K$, $g \in G$
such that $k-g \in X \times \{0\}$ must be of the form
$k = (\lambda c, \lambda)$, $g =  (\lambda d,  \lambda)$ for $c \in C$,
$d \in D$, and
$\lambda \ge 0$. Then
\begin{equation*}
k-g =  (\lambda(c-d), 0) \in \cone(C-D) \times \{0\}.
\end{equation*}
\end{proof}
\begin{proof}[\hypertarget{loc:minima_dehomogenize_to_the_induced_order.proof}{}Proof of \Cref{loc:minima_dehomogenize_to_the_induced_order.statement}]

For $c, c' \in C$, the difference of the lifts is a horizon vector, $\gamma_1(c') - \gamma_1(c) = \gamma_0(c'-c)$. Therefore 
$\gamma_1(c')-\gamma_1(c)  \in P  \iff \gamma_0(c' - c)  \in P$, i.e., $\gamma_1(c) \leq_{P} \gamma_1(c') \iff c \leq_{\gamma_0^{-1}(P)} c'$.

$(\subseteq)$: If $\gamma_1(c) \le_{P} K$ then in particular $\gamma_1(c) \le_{P} \gamma_1(c')$ for all $c' \in C$, and therefore $c \le_{\gamma_0^{-1}(P)} C$.

$(\supseteq)$: Let $x_0 \in \gamma_1^{-1}(\min_P K)$ and $c \in \min_{\gamma_0^{-1}(P)}C$. Then since $x_{0} \in C$, we have $c \le_{\gamma_0^{-1}(P)} x_0$, so $\gamma_1(c) \le_{P} \gamma_1(x_0)$, and $\gamma_1(x_0) \le_{P} K$. Transitivity gives $\gamma_1(c) \in \min_P K$.
\end{proof}
\begin{proof}[\hypertarget{loc:the_joint_supporting_subspace_is_the_generated_supporting_flat.proof}{}Proof of \Cref{loc:the_joint_supporting_subspace_is_the_generated_supporting_flat.statement}]

The homogenization $G := \cone(\gamma_1(S))$ is a subcone of $K := \cone(\gamma_1(C))$, so \Cref{loc:separation_reduces_to_containment.statement} point 2 makes $\lin(K-G)$ the minimum supporting subspace of $K$ containing $G$. We dehomogenize that minimality.

The subcone $G$ dehomogenizes back to $S$, by \Cref{loc:dehomogenization_is_a_lattice_isomorphism_between_convex_sets_and_homogenization_cones.statement}, and it is generated by $\gamma_1(S)$, so a subspace $U \subseteq X \times \mathbb{R}$ satisfies
\begin{equation*}
G \subseteq U \iff S \subseteq \gamma_1^{-1}(U).
\end{equation*}
By \Cref{loc:dehomogenization_of_kernels.statement}, $\gamma_1^{-1}$ carries the supporting subspaces of $K$ surjectively onto the supporting flats of $C$, so by the displayed equivalence it carries those containing $G$ surjectively onto those containing $S$. Preserving inclusion, it carries minimum to minimum: the minimum supporting flat of $C$ containing $S$ is $\gamma_1^{-1}(\lin(K-G))$, which is $T_a(C,S)$ by \Cref{loc:affine_cross:lineality.statement} point 1.
\end{proof}
\begin{proof}[\hypertarget{loc:alternatives_for_a_separating_affine_oriented_map.proof}{}Proof of \Cref{loc:alternatives_for_a_separating_affine_oriented_map.statement}]

Let $K := \cone(C \times \{1\})$ and $G := \cone(D \times \{1\})$ in $X \times \mathbb{R}$, and let $\phi:X \times \mathbb{R} \to (Y, Y_+)$ be the oriented linear map associated with $f$ by \Cref{loc:oriented_linear_maps_are_in_bijection_with_oriented_affine_maps.statement}, namely $\phi(x, \lambda) = f(x) - (1-\lambda)f(0)$. By construction $\phi(G) \le 0 \le \phi(K)$.

Apply \Cref{loc:alternatives_for_a_separating_linear_oriented_map.statement} to $\phi$. The cases dehomogenize as follows.
For the first case of \Cref{loc:alternatives_for_a_separating_linear_oriented_map.statement}
dehomogenize to the first case of \Cref{loc:alternatives_for_a_separating_affine_oriented_map.statement}: if $K \cap \ker(\phi) = \{0\}$, then $C \cap \ker f=\varnothing$, and similarly for $G$ and $D$. With $f(D) \leq f(C)$, this is enough to guarantee that $f(D) < f(C)$, because given two points $c  \in C, d  \in D$
such that $f(c)=f(d)$ we would have $0 \leq f(c) =f(d) \leq 0$, so $c,d  \in \ker f$, contradicting the fact that at least one of $K \cap \ker \phi$, $G \cap \ker \phi$ is empty.

For the second case, $\ker f = \gamma_1^{-1}(\ker \phi) = \gamma_1^{-1}(\lin(K-G)) = \lin\cone(C-D) + C \cap D$ by \Cref{loc:affine_cross:lineality.statement} points 1 and 2.

For the third case, the bijection of \Cref{loc:oriented_linear_maps_are_in_bijection_with_oriented_affine_maps.statement} dehomogenizes the linear $\psi_h:X \times \mathbb{R} \to (\mathbb{R}, \mathbb{R}_+)$ to the affine $\psi:X \to (\mathbb{R}, \mathbb{R}_+)$ given by $\psi(x) := \psi_h(x, 1)$. Pulling back by $\gamma_1$ acts coordinate-wise, so it sends $\phi \rtimes \psi_h$ to $f \rtimes \psi$, transferring the separation to $(f \rtimes \psi)(D) \le 0 \le (f \rtimes \psi)(C)$.

Still in the third case, we argue that $\psi$ is not constant. Suppose case 1 fails, so that $K \cap \ker\phi = \cone((C \cap \ker f) \times \{1\})$ and $G \cap \ker\phi = \cone((D \cap \ker f) \times \{1\})$ are non-trivial homogenization cones. If $\psi$ were constant on $\ker f$ with value $c$, then on these cones $\psi_h(\lambda x, \lambda) = \lambda \psi_h(x, 1) = \lambda \psi(x) = \lambda c$ for $x \in C \cap \ker f$ or $D \cap \ker f$, so $\psi_h(K \cap \ker\phi) = \psi_h(G \cap \ker\phi) = \{\lambda c : \lambda \ge 0\}$. The separation $\psi_h(G \cap \ker\phi) \le 0 \le \psi_h(K \cap \ker\phi)$ forces this common set into $\{0\}$, so $c = 0$; combined with $\psi_h$ vanishing on $\{(x_1 - x_2, 0) : x_1, x_2 \in \ker f\}$ (by linearity of $\psi_h$ and $\psi(x_1) = \psi(x_2) = 0$), $\psi_h$ vanishes on a generating set of $\ker\phi$, contradicting $\psi_h \neq 0$ on $\ker\phi$. Hence $\psi$ is non-constant on $\ker f$.
\end{proof}
\end{document}

%% file: header.tex
\usepackage{amsmath}
\usepackage{amsthm}
\usepackage{aliascnt}
\usepackage{biblatex}
\usepackage{graphicx}
\usepackage{tikz-cd}

\theoremstyle{plain}
\newtheorem{theorem}{Theorem}[section]

\newaliascnt{corollary}{theorem}
\newtheorem{corollary}[corollary]{Corollary}
\aliascntresetthe{corollary}

\newaliascnt{lemma}{theorem}
\newtheorem{lemma}[lemma]{Lemma}
\aliascntresetthe{lemma}

\newaliascnt{proposition}{theorem}
\newtheorem{proposition}[proposition]{Proposition}
\aliascntresetthe{proposition}

\theoremstyle{definition}
\newaliascnt{definition}{theorem}
\newtheorem{definition}[definition]{Definition}
\aliascntresetthe{definition}

\theoremstyle{remark}
\newaliascnt{remark}{theorem}
\newtheorem{remark}[remark]{Remark}
\aliascntresetthe{remark}

\newaliascnt{fact}{theorem}

\aliascntresetthe{fact}

\usepackage{hyperref}
\usepackage{cleveref}

\crefname{corollary}{Corollary}{Corollaries}
\Crefname{corollary}{Corollary}{Corollaries}
\crefname{lemma}{Lemma}{Lemmas}
\Crefname{lemma}{Lemma}{Lemmas}
\crefname{proposition}{Proposition}{Propositions}
\Crefname{proposition}{Proposition}{Propositions}
\crefname{definition}{Definition}{Definitions}
\Crefname{definition}{Definition}{Definitions}
\crefname{remark}{Remark}{Remarks}
\Crefname{remark}{Remark}{Remarks}
\crefname{fact}{Fact}{Facts}